\documentclass[onefignum,onetabnum]{siamart190516}

\usepackage{amsfonts}
\usepackage{amssymb}
\usepackage{bm}
\usepackage{enumitem}
\usepackage[normalem]{ulem}

\newsiamremark{remark}{Remark}
\newsiamremark{hypothesis}{Hypothesis}
\newsiamremark{notation}{Notation}
\crefname{hypothesis}{Hypothesis}{Hypotheses}
\crefname{assumption}{Assumption}{Assumptions}
\crefalias{enumi}{assumption}
\newsiamthm{claim}{Claim}
\newsiamthm{problem}{Problem}
\newsiamthm{modelproblem}{Model problem}

\headers{Optimal control of CDR problems using IGA}{Kent-Andre Mardal, Jarle Sogn, and Stefan Takacs}

\title{Robust preconditioning and error estimates for optimal control of the convection-diffusion-reaction equation with limited observation in Isogeometric analysis\thanks{\funding
{The first author acknowledges support from the Research Council of Norway, grant 300305 and 301013. The second and the third author are supported by the Austrian Science Fund (FWF): P31048.}}} 

\author{Kent-Andre Mardal\thanks{Department of Mathematics, University of Oslo, Oslo, Norway. Center for Biomedical Computing, Simula Research Laboratory, Lysaker, Norway (\email{kent-and@simula.no}).}
\and Jarle Sogn\thanks{Johann Radon Institute for Computational Mathematics (RICAM), Austrian Academy of Sciences, Linz, Austria 
  (\email{jarle.sogn@ricam.oeaw.ac.at}).}
\and Stefan Takacs\thanks{Johann Radon Institute for Computational Mathematics (RICAM), Austrian Academy of Sciences, Linz, Austria 
  (\email{stefan.takacs@ricam.oeaw.ac.at}).}}

\usepackage{amsopn}

\newcommand{\inner}[2]{\left(#1,#2\right)}
\newcommand{\dual}[2]{\left<#1,#2\right>}
\newcommand{\LLO}{L^2(\Omega)}

\newcommand{\foralls}{\forall \,}

\newcommand{\commentJS}[1]{{\color{red}{\bf Jarle:} #1}}
\definecolor{stcolor}{rgb}{0.5, 0.0, 0.6}

\newcommand{\commentST}[1]{{\color{stcolor}{\bf Stefan:} #1}}

\ifpdf
\hypersetup{
  pdftitle={title},
  pdfauthor={Kent-Andre Mardal, Jarle Sogn, and Stefan Takacs}
}
\fi

\usepackage{etoolbox}
\makeatletter
\patchcmd{\@addmarginpar}{\ifodd\c@page}{\ifodd\c@page\@tempcnta\m@ne}{}{}
\makeatother
\reversemarginpar

\begin{document}

\maketitle

\begin{abstract}
In this paper we analyze an optimization problem with limited observation governed by a convection--diffusion--reaction equation. Motivated by a Schur complement approach, we
arrive at continuous norms that enable analysis of well-posedness
and subsequent derivation of error analysis and a preconditioner that is robust with respect to the parameters of the problem.
We provide conditions for inf-sup stable discretizations and present one such discretization for box domains with constant convection. We also provide a priori error estimates for this discretization. The preconditioner requires a fourth order problem to be solved. For this reason, we use Isogeometric Analysis as a method of discretization.  To efficiently realize the preconditioner, we consider geometric multigrid with a standard Gauss-Seidel smoother as well as a new macro Gauss-Seidel smoother. The latter smoother provides good results with respect to both the geometry mapping and the polynomial degree.
\end{abstract}


\section{Introduction}

In this paper, we consider an optimal control problem involving a linear Convection--Diffusion--Reaction (CDR) problem, which reads as follows:
  \begin{equation}
    \label{prob:main}
\text{Minimize} \quad J(u,q) :=  \frac12 \|u-u_d \|^2_{L^2(\mathcal{O})} + \frac{\alpha}{2} \| q \|^2_{L^2(\Omega)} \quad \text{for}\quad u\in U, q\in \LLO
  \end{equation}
  subject to 
  \begin{align}
    \label{eq:stateEq}
    \begin{split}
    - \varepsilon \Delta u + \beta\cdot \nabla u  + \sigma u  &= f-q \quad \text{in} \quad\Omega,\\
    u &= 0 \quad \quad \:\:\: \text{on} \quad \partial\Omega.
    \end{split}
  \end{align}
  Here and in what follows, $\Omega$ is a bounded open subset of $\mathbb{R}^d$
  ($d=1,2,3$) with Lipschitz boundary, 
  $f\in \LLO$, $\varepsilon, \sigma \in \mathbb{R}$ with $\varepsilon > 0, \sigma \geq 0$,
  $\beta \in L^{\infty}(\Omega)^d$ with $\nabla \cdot \beta = 0$
  and $\mathcal O\subseteq \Omega$ is measurable in $\mathbb{R}^d$.
  For certain choices of the parameters, like $\varepsilon \ll \beta, \sigma$, convection--diffusion--reaction problems are singular perturbation problems exhibiting sharp gradients and a potential for loss of regularity.
  To overcome the problems associated to the loss of regularity, significant effort has been put in the development of methods with low regularity,
  such as discontinuous Galerkin methods~\cite{baumann1999discontinuous, cockburn1998local}.
  We take the opposite approach and investigate to what extent higher regularity may be used in the setting of optimal control problems. 
  
  There are two main problems with \eqref{prob:main}--\eqref{eq:stateEq}, namely: 1) potential sharp gradients leading to non-physical oscillations in the numerical solution and 2) ill-posedness due to limited observations, this is,  when $\mathcal{O}$ is a subset of $\Omega$.
  Motivated by the fact that higher regularity has been exploited in the cases with limited observations~\cite{MarNieNor17,SogZul18,beigl2019robust}, we derive order optimal preconditioners via stability analysis in non-standard
  Sobolev spaces.

  When solving the CDR problem, it is common to consider some stabilization method (e.g. the streamline upwind Petrov Galerkin (SUPG) method) or adaptive grids (e.g. Shishkin grids) to reduce the oscillatory behavior, cf. \cite{elman2014finite}.
  This is also the case in optimal control settings, see, e.g.,  \cite{quarteroni2005optimal, becker2007optimal,hinze2009variational,chen2018hdg}.
  We do not use any such stabilization techniques, but we remark that the trial and test functions involved in the state equation differ.
  That means, the state equation, if considered isolated, is discretized by a Petrov-Galerkin method, although the complete optimality system is discretized  by a standard Galerkin method, this is, trial and test functions agree.
  In particular, in the continuous setting, the trial space and test spaces are  $H^2(\Omega)\cap H^1_0(\Omega)$ and $L^2(\Omega)$, respectively,
  with properly weighted norms.
  
  By considering a Schur complement of the optimal control problem,
  we derive non-standard norms in which well-posedness is obtained, assuming extra regularity. From the well-posedness of the continuous system we subsequently analyse corresponding discrete systems to arrive at both error estimates and preconditioners that are robust with respect to the problem parameters  $\alpha, \varepsilon,\beta$ and $\sigma$. In detail, we provide a condition for the discretization which ensures that the preconditioner is sparse and that the preconditioned system is stable. Further, we give an example of such a discretization based on Isogeometric Analysis (IgA) \cite{HugCotBaz05,VeiBufSanVaz14}. For our approach, IgA provides useful discretization methods since the extra regularity leads  $H^2(\Omega)$--conforming approximation spaces. Using these discretization methods, a priori error estimates are derived, where we detail the dependencies of the problem parameters. We remark that the
  error estimates required extending some approximation error estimates for tensor-product B-splines, which is done in Appendix \ref{sec:app2}.    
  
  Similar Schur complement preconditioners were used, on the linear algebra level, in \cite{PeaWat12, porcelli2015preconditioning} for optimal control problems of the CDR equation. \cite{porcelli2015preconditioning} also considers mixed constraints. The preconditioners perform well for different values of the problem parameters. The Schur complement preconditioners were replaced with approximations based on the factorization approach by \cite{PeaWat12}. However, this approach does not work well for problems with limited observation, i.e., when $\mathcal{O}\subsetneq\Omega$. 
  
  To solve the resulting linear system we use preconditioned Krylov subspace methods. 
  We consider two approaches to realize our preconditioner: sparse direct methods and multigrid methods. For mid-sized problems, sparse direct solvers work well since each component of the preconditioner is symmetric and positive definite. For large-sized problems, we use a multigrid method to realize the fourth-order operator.
  Combining the results from this paper and from \cite{sogn2018schur,sogn2019robust}, it follows that the multigrid method we consider is robust in the grid-size, however, not necessarily in any of the other problem parameters. Finding a multigrid method which is robust in the grid-size, the chosen spline degrees, $\alpha, \varepsilon,\beta, \sigma$ and $\mathcal{O}$, remains an open problem.
 
  The outline of the paper is as follows: in the next section we perform the analysis of the continuous problem. In Section~\ref{sec:discAna}, we analyse the discrete problem and provide a condition for a stable discretization. IgA is then introduced in Section~\ref{sec:IGA} along with the proposed discretization. 
  In Section~\ref{sec:errorEst}, a priori error estimates are derived. Section~\ref{sec:NumCanonical} contains a discussion of the solution of an one-dimensional problem and in Section~\ref{sec:num} we perform numerical experiments on two-dimensional problems and see how the preconditioner behaves.

\section{Analysis of the continuous problem}

To obtain a standard (weak) variational formulation of the state equation \eqref{eq:stateEq}, one would choose the state variable $u$ and test function $\tilde{w}$ to be in $H^1_0(\Omega)$. Instead, we consider the strong variational formulation, this is,
find $u\in H^2(\Omega)\cap H^1_0(\Omega)$ such that
\begin{align*}
    (q,\tilde{w})_{L^2(\Omega)} + (- \varepsilon \Delta u +\beta\cdot \nabla u + \sigma u,\tilde{w})_{L^2(\Omega)}  &= (f,\tilde{w})_{\LLO} \quad &\forall& \, \tilde{w}\in L^2(\Omega).
\end{align*}
The Lagrangian functional associated to (\ref{prob:main})--(\ref{eq:stateEq}) is
\begin{align*}
    \mathcal{L}(q,w,u) &:= \frac12 \|u-u_d \|^2_{L^2(\mathcal{O})} + \frac{\alpha}{2} \| q \|^2_{L^2(\Omega)} \\
    &+ (q,w)_{L^2(\Omega)} + (- \varepsilon \Delta u +\beta\cdot \nabla u + \sigma u,w)_{L^2(\Omega)}  - (f,w)_{\LLO},
\end{align*}
where $u \in H^2(\Omega)\cap H^1_0(\Omega)$, $q \in L^2(\Omega)$ and the Lagrangian multiplier $w \in L^2(\Omega)$. From the first order necessary optimality conditions
\[
 \frac{\partial \mathcal{L}}{\partial q}(q,w,u)=0,\quad \frac{\partial \mathcal{L}}{\partial w}(q,w,u)=0,\quad \frac{\partial \mathcal{L}}{\partial u}(q,w,u)=0,\quad
\]
which are also sufficient here, we obtain the optimality system: 

\begin{problem}
  \label{prob:KKTCont}
  Find $(q,w,u)\in \LLO\times \LLO \times H^2(\Omega)\cap H^1_0(\Omega)$ such that
\begin{align*}
  \alpha (q,\tilde{q})_{L^2(\Omega)}  + (w,\tilde{q})_{L^2(\Omega)}   &= 0 \;\;&\forall& \, \tilde{q}\in L^2(\Omega), \\
  (q,\tilde{w})_{L^2(\Omega)} + (- \varepsilon \Delta u +\beta\cdot \nabla u + \sigma u,\tilde{w})_{L^2(\Omega)}  &= (f,\tilde{w})_{\LLO} \;\; &\forall& \, \tilde{w}\in L^2(\Omega), \\
 (w, - \varepsilon \Delta \tilde{u}+ \beta\cdot \nabla \tilde{u} + \sigma \tilde{u}  )_{L^2(\Omega)} + (u,\tilde{u})_{L^2(\mathcal{O})}&= (u_d,\tilde{u})_{L^2(\mathcal{O})} \;\; &\forall& \, \tilde{u}\in H^2(\Omega)\cap H^1_0(\Omega).
\end{align*}
\end{problem}
Problem~\ref{prob:KKTCont} can be written as
\begin{equation}
  \label{eq:MatCon1} 
  \mathcal{A}
  \begin{pmatrix}
    q\\
    w\\
    u
  \end{pmatrix} =
    \begin{pmatrix}
    0\\
    Mf\\
    \tilde{M}_{\mathcal{O}}u_d
    \end{pmatrix}
\quad \text{with} \quad \mathcal{A} := 
    \begin{pmatrix}
    \alpha M  & M & 0 \\
    M & 0 & K \\
    0 & K' & M_{\mathcal{O}} \\
  \end{pmatrix}.
\end{equation}
Here,  $M: L^2(\Omega) \rightarrow (L^2(\Omega))'$
represents the $L^2(\Omega)$-inner product, that is, we have
\[
\dual{M q}{w} = (q,w)_{L^2(\Omega)},
\]
where $\dual{\cdot}{\cdot}$ denotes the duality product. The notation "$\,'$" is used to denote both dual spaces and dual operators. $K: H^2(\Omega)\cap H^1_0(\Omega) \rightarrow (\LLO)'$ is the state operator:
\[
\dual{Ku}{w} = (-\varepsilon \Delta u +\beta\cdot \nabla u + \sigma u,w)_{\LLO}.
\]
Finally,  $M_{\mathcal{O}}: H^2(\Omega)\cap H^1_0(\Omega) \rightarrow (H^2(\Omega)\cap H^1_0(\Omega))'$ and $\tilde{M}_{\mathcal{O}}: L^2(\mathcal{O}) \rightarrow (H^2(\Omega)\cap H^1_0(\Omega))'$, and both represent the $L^2(\mathcal{O})$-inner product on the subdomain $\mathcal{O}$.

We observe that the block operator $\mathcal{A}$ has a block tridiagonal form. Such tridiagonal operators are studied in \cite{SogZul18, beigl2019robust}. We use the Schur complement preconditioner proposed in \cite{SogZul18}:
\begin{equation}
\label{eq:SchurCont}
  \mathcal{S}(\mathcal{A}) :=     \begin{pmatrix}
    S_q  & 0 & 0 \\
    0 & S_w & 0 \\
    0 & 0 & S_u \\
  \end{pmatrix},
\end{equation}
where the components are
\begin{equation}
  \label{eq:schurCompls}
  S_q := \alpha M, \quad S_w := \frac{1}{\alpha} M,  \quad S_u :=  M_{\mathcal{O}}+ \alpha K'M^{-1}K.
\end{equation}
These Schur complements define weighted norms as follows:
\begin{align}
\label{eq:SchurNorm}
\begin{split}
  \|q\|^2_{S_q} &:= \dual{S_q q}{q} =  \alpha \|q\|^2_{\LLO},\\
  \|w\|^2_{S_w} &:= \dual{S_w w}{w} =  \frac{1}{\alpha} \|w\|^2_{\LLO},\\
  \|u\|^2_{S_u} &:=\dual{S_u u}{u} =  \|u\|^2_{L^2(\mathcal{O})} +\alpha \|- \varepsilon \Delta u +\beta\cdot \nabla u + \sigma u\|^2_{\LLO}.
  \end{split}
    \end{align}
The last norm follows from 
\begin{align*}
    \dual{K' M^{-1}K u}{u}  &= \sup_{0\neq w \in L^2(\Omega)}\frac{\dual{Ku}{w}^2}{\dual{Mw}{w}} =\sup_{0\neq w \in L^2(\Omega)} \frac{(-\varepsilon \Delta u +\beta\cdot \nabla u + \sigma u,w)^2_{\LLO}}{\|w\|^2_{\LLO}}\\
    & =\frac{\|- \varepsilon \Delta u +\beta\cdot \nabla u + \sigma u\|^4_{\LLO}}{\|-\varepsilon \Delta u +\beta\cdot \nabla u + \sigma u\|^2_{\LLO}} = \|- \varepsilon \Delta u +\beta\cdot \nabla u + \sigma u\|^2_{\LLO}.
\end{align*}
We show well-posedness with respect to the norms (\ref{eq:SchurNorm}) by showing that the operator 
\begin{equation}
  \label{eq:Aiso}
\mathcal{A}: \LLO\times \LLO \times H^2(\Omega)\cap H^1_0(\Omega) \rightarrow \LLO'\times \LLO' \times (H^2(\Omega)\cap H^1_0(\Omega))'
\end{equation}
is an isomorphism with respect to the norms (\ref{eq:SchurNorm}). 
This is done by using the main result in \cite{SogZul18}, which for our problem reads as follows.
\begin{theorem}
  \label{theo:schurPos}
  Assume the Schur complements in (\ref{eq:schurCompls}) are well defined and positive definite, that is,
  \begin{equation}
    \label{eq:positiveSchur}
    \dual{S_q q}{q} \geq \sigma_q \|q\|^2_{L^2(\Omega)},\quad \dual{S_w w}{w} \geq \sigma_w \|w\|^2_{L^2(\Omega)}, \quad \dual{S_u u}{u} \geq \sigma_u \|u\|^2_{H^2(\Omega)}
  \end{equation}
  for some positive constants $\sigma_q$, $\sigma_w$ and $\sigma_u$, which can depend on the given parameters. Then, $\mathcal{A}$ in (\ref{eq:Aiso}) is an isomorphism, moreover, the condition number of the preconditioned operator $\mathcal{S}^{-1}\mathcal{A}$ is bounded:
  \begin{equation}
  \label{eq:condNumber}
      \kappa\left(\mathcal{S}^{-1}\mathcal{A}\right) \leq \frac{\cos(\pi/7)}{\sin(\pi/14)} \approx 4.05.
  \end{equation}
  \end{theorem}
The conditions in (\ref{eq:positiveSchur}) ensure that the spaces $L^2(\Omega)$, $L^2(\Omega)$ and $H^2(\Omega)\cap H^1_0(\Omega)$, equipped with norms $\|\cdot\|_{S_q}$, $\|\cdot\|_{S_w}$ and $\|\cdot\|_{S_u}$, are complete.
Before proving Condition \eqref{eq:positiveSchur}, we provide a useful lemma which bounds the $H^2$-norm. The proof of this lemma is presented in Appendix \ref{sec:proof}.
\begin{lemma}
\label{lemma:SFI}
If the domain $\Omega$ has a Lipschitz boundary and
\begin{itemize}
    \item the boundary is a polygon (polyhedron) or
    \item the domain is the image of a geometric mapping $\mathbf{G}:\widehat{\Omega}:=(0,1)^d\rightarrow \Omega$, where both
$\|\nabla^r \mathbf{G}\|_{L^\infty(\widehat{\Omega})}$ and
$\|(\nabla^r \mathbf{G})^{-1}\|_{L^\infty(\widehat{\Omega})}$ are bounded for $r\in\{1,2,3\}$,
\end{itemize}
then the $H^2$-norm is bounded by the $L^2$-norm of the Laplacian, i.e., 
\begin{equation}
\label{eq:2FAPE}
\|u\|_{H^2(\Omega)} \leq C_\Omega \|\Delta u\|_{L^2(\Omega)}\quad \foralls u\in H^2(\Omega) \cap H^1_0(\Omega),
\end{equation}
for a constant $C_\Omega$ depending only on $\Omega$.
\end{lemma}
\begin{theorem}
\label{theo:contMain}
If $\Omega$ is a domain such that the conditions of Lemma~\ref{lemma:SFI} hold, then the assumptions of Theorem~\ref{theo:schurPos} are satisfied for Problem~\ref{prob:KKTCont}.
\end{theorem}
\begin{proof}
For simplicity, we prove this lemma only for $\sigma = 0$. An extension to the case $\sigma >0$ is straight-forward.

The first two conditions are trivial since $\dual{S_q q}{q} = \alpha \|q\|^2_{L^2(\Omega)}$ and \\$\dual{S_w w}{w} =\frac{1}{\alpha} \|w\|^2_{L^2(\Omega)}$. For the third condition, let 
  \[
  \delta := \frac{\|\beta\| }{\frac{\varepsilon}{c_P} +\|\beta\| },
  \]
  where $\|\beta\| = \|\beta\|_{L^{\infty}(\Omega)}$ and $c_P$ is constant from the Poincar\'e inequality, that is, we have
  \[
  \| u\|_{L^2(\Omega)} \leq c_p\|\nabla u\|_{L^2(\Omega)}.
  \] Note that $\delta \in [0,1)$. Let $u\in H^2(\Omega)\cap H^1_0(\Omega)$ be arbitrary but fixed. We consider the two cases:
  \begin{equation}
    \label{eq:casesProof}
  \|\beta\|  \|\nabla u \|_{L^2(\Omega)} < \delta \varepsilon \|\Delta u\|_{L^2(\Omega)} \quad \text{and} \quad   \|\beta\|  \|\nabla u \|_{L^2(\Omega)} \geq \delta \varepsilon \|\Delta u\|_{L^2(\Omega)}.
  \end{equation}
  \emph{First case.} From the definition, we have
  \[
  \dual{K'M^{-1}K u}{u} = \sup_{0\neq w \in L^2(\Omega)} \frac{\inner{\beta \cdot \nabla u - \varepsilon \Delta u}{w}^2_{L^2(\Omega)}}{\|w\|^2_{L^2(\Omega)}}.
  \]
  By setting $w = -\varepsilon \Delta u$, we get using the
  Cauchy--Schwarz inequality that
  \begin{align*}
    \dual{K'M^{-1}K u}{u} &\geq \frac{(\inner{\beta \cdot \nabla u}{- \varepsilon \Delta u}_{L^2(\Omega)} + \varepsilon^2 \|\Delta u\|^2_{L^2(\Omega)})^2}{\varepsilon^2 \|\Delta u\|^2_{L^2(\Omega)}}\\
    &\geq \frac{(-\| \beta\| \| \nabla u\|_{L^2(\Omega)} \varepsilon \|\Delta u\|_{L^2(\Omega)} + \varepsilon^2 \|\Delta u\|^2_{L^2(\Omega)})^2}{\varepsilon^2 \|\Delta u\|^2_{L^2(\Omega)}}\\
    &= (-\| \beta\| \| \nabla u\|_{L^2(\Omega)}  + \varepsilon \|\Delta u\|_{L^2(\Omega)})^2.
  \end{align*}
  Using the first inequality in (\ref{eq:casesProof}), we obtain
  \begin{align*}
    \dual{K'M^{-1}K u}{u} \geq (-\| \beta\| \| \nabla u\|_{L^2(\Omega)}  + \varepsilon \|\Delta u\|_{L^2(\Omega)})^2 &\geq ((1-\delta)\varepsilon \|\Delta u\|_{L^2(\Omega)})^2\\
    &= \frac{\varepsilon^4}{(\varepsilon +c_P\|\beta\| )^2}  \|\Delta u\|^2_{L^2(\Omega)}.
  \end{align*}
  \emph{Second case.} By setting $w = u$, we get
\begin{align*}
    \dual{K'M^{-1}K u}{u} &\geq \frac{(\inner{\beta \cdot \nabla u}{ u}_{L^2(\Omega)} + \varepsilon \|\nabla u\|^2_{L^2(\Omega)})^2}{\| u\|^2_{L^2(\Omega)}} = \frac{\varepsilon^2 \|\nabla u\|^4_{L^2(\Omega)}}{\| u\|^2_{L^2(\Omega)}}
  \end{align*}
using integration by parts. Due to the homogeneous Dirichlet boundary conditions and $\nabla \cdot \beta = 0$, the term $\inner{\beta \cdot \nabla u}{ w}_{L^2(\Omega)}$ is skew symmetric and vanishes for $w = u$.
Finally, we use $\| u\|_{L^2(\Omega)} \leq c_p\|\nabla u\|_{L^2(\Omega)}$  and the second inequality in (\ref{eq:casesProof}), which gives
\begin{align*}
\dual{K'M^{-1}K u}{u} &\geq  \frac{\varepsilon^2 \|\nabla u\|^4_{L^2(\Omega)}}{\| u\|^2_{L^2(\Omega)}} \geq \frac{\varepsilon^2}{c^2_P} \|\nabla u\|^2_{L^2(\Omega)}\\ &\geq
\frac{\delta^2\varepsilon^4}{\|\beta\|^2 c^2_P} \|\Delta u\|^2_{L^2(\Omega)} =  \frac{\varepsilon^4}{(\varepsilon + c_P\|\beta\| )^2}\|\Delta u\|^2_{L^2(\Omega)}.
\end{align*}
  To summarize, in both cases we get 
  \begin{align*}
  \dual{S_u u}{u} &= \|u\|^2_{L^2(\mathcal{O})}+ \alpha \dual{K'M^{-1}K u}{u}\geq \alpha\dual{K'M^{-1}K u}{u} \\ 
    &\geq   \alpha\frac{\varepsilon^4}{(\varepsilon + c_P\|\beta\| )^2}  \|\Delta u\|^2_{L^2(\Omega)} \geq \alpha\frac{C_\Omega\varepsilon^4}{(\varepsilon + c_P\|\beta\| )^2}  \| u\|^2_{H^2(\Omega)}.
  \end{align*}
  The last inequality follows from Lemma~\ref{lemma:SFI}.
  \end{proof}

Theorem~\ref{theo:schurPos} and Theorem~\ref{theo:contMain} show that Problem~\ref{prob:KKTCont} is well-posed with respect to the norms in (\ref{eq:SchurNorm}). The boundedness and coercivity constants are bounded independent from the regularization parameter $\alpha$ as well as the problem parameters $\varepsilon$, $\beta$ and $\sigma$.
Consequently, the operator preconditioner (\ref{eq:SchurCont}) is a robust preconditioner for the optimality system, that is, the condition number is uniformly bounded independently
of the above mentioned parameters. 
So far, we have only analyzed the problem on the continuous level. In the next section, we carry this analysis over to the discrete case and provide a
computationally feasible preconditioner.

\section{Analysis of the discrete problem}
\label{sec:discAna}
We consider conforming discretizations,
that is, we choose the finite-dimensional
spaces $Q_h$ and $U_h$ such that they satisfy
\begin{equation*}
Q_h \subset L^2(\Omega) \quad \text{and}\quad U_h \subset H^2(\Omega)\cap H^1_0(\Omega).
\end{equation*}
Applying Galerkin’s principle to (\ref{eq:MatCon1}) leads to the
discrete variational problem for the functions $(q_h,w_h,u_h)\in Q_h \times Q_h \times U_h$, which we immediately write in matrix-vector
notation. We denote the vector representation of functions in these spaces by underlined versions of the
corresponding symbols, i.e., for $q_h \in Q_h$ the corresponding coefficient
vector is $\underline q_h \in \mathbb R^{\dim Q_h}$. 
With a slight abuse of notation, we use the same notation also
for the right-hand-side vectors $\underline f_h$ and $\underline u_{d,h}$,
which are obtained by testing the corresponding linear functionals with the
basis functions in $Q_h$ and $U_h$, respectively, see, e.g., 
\cite[Section~6]{mardal2011preconditioning} for further details.
Furthermore, 
operators with subscript $h$ denote matrix representations
of the operators.

Using
this notation, the discrete problem reads as follows.
\begin{problem}
\label{prob:disc}
  Find $(\underline q_h,\underline w_h,\underline u_h)\in \mathbb R^{\dim Q_h} \times \mathbb R^{\dim Q_h} \times \mathbb R^{\dim U_h}$ such that 
  \begin{equation}
  \label{eq:MatDis1} 
  \mathcal{A}_h
  \begin{pmatrix}
    \underline{q}_h\\
    \underline{w}_h\\
    \underline{u}_h
  \end{pmatrix} =
    \begin{pmatrix}
    0\\
    \underline{f}_h\\
    \underline{u}_{d,h}
    \end{pmatrix}
    \quad\mbox{with}\quad
    \mathcal{A}_h :=
 \begin{pmatrix}
    \alpha M_h  & M_h & 0 \\
    M_h & 0 & K_h \\
    0 & K^T_h & M_{\mathcal{O},h} \\
  \end{pmatrix}.
\end{equation}
\end{problem}
The exact Schur complement preconditioner (\ref{eq:SchurCont}) of the discretized system is 
\begin{equation}
\label{eq:SchurExactDisc}
  \mathcal{S}(\mathcal{A}_h) =     \begin{pmatrix}
    \alpha M_h  & 0 & 0 \\
    0 & \frac{1}{\alpha}M_h & 0 \\
    0 & 0 & M_{\mathcal{O},h} + \alpha K^T_h M^{-1}_h K_h \\
  \end{pmatrix}.
\end{equation}
Under the mild condition $U_h \subseteq Q_h$ this preconditioner is symmetric positive definite. This is a straight forward extension of \cite[Lemma 4.4]{SogZul18} by using the fact that $(\beta\cdot\nabla u_h,u_h)_{L^2(\Omega)} = 0$. In this case, Theorem~\ref{theo:schurPos} yields the following condition number bound:
\begin{equation*}
    \kappa\left((\mathcal{S}(\mathcal{A}_h))^{-1}\mathcal{A}_h\right) \leq  \frac{\cos (\pi/7)}{\sin (\pi/14)}
    \approx 4.05\, .
\end{equation*}
This preconditioner cannot be efficiently realized since the matrix $M^{-1}_h$ is dense. So, we use the following preconditioner motivated by the norms (\ref{eq:SchurNorm}) on the discretization spaces instead:
\begin{equation}
\label{eq:SchurNormDisc}
  \mathcal{S}_h :=     \begin{pmatrix}
    \alpha M_h  & 0 & 0 \\
    0 & \frac{1}{\alpha}M_h & 0 \\
    0 & 0 & M_{\mathcal{O},h} + \alpha B_h \\
  \end{pmatrix},
\end{equation}
where $B_h$ is the matrix representation of the linear operator $B: H^2(\Omega)\cap H^1_0(\Omega) \rightarrow (H^2(\Omega)\cap H^1_0(\Omega))'$,
\begin{equation}
    \dual{B u}{\tilde{u}} = (- \varepsilon \Delta u +\beta\cdot \nabla u + \sigma u,  - \varepsilon \Delta \tilde{u}+\beta\cdot \nabla \tilde{u}+ \sigma \tilde{u})_{\LLO}
\end{equation}
on $U_h$. On the continuous level, the operators $B$ and $K' M^{-1} K$ coincide. 
In general, this does not carry over to the discrete case. 
The following lemma gives sufficient conditions that guarantee
that $B_h$ and $K^T_h M^{-1}_h K_h$ coincide.
\begin{lemma}
\label{lemma:conformingDisc}
If
\begin{equation}
\label{eq:conformingDisc}
    (-\varepsilon\Delta +\beta\cdot\nabla+ \sigma)U_h \subset Q_h,
\end{equation}
then $K^T_h M^{-1}_h K_h=B_h$ and thus
$\mathcal{S}(\mathcal{A}_h)=\mathcal{S}_h$.
\end{lemma}
\begin{proof}
Let $u_h\in U_h$ be arbitrary but fixed with coefficient
vector $\underline u_h$. The definitions yield
\begin{align*}
    \dual{K^T_h M^{-1}_h K_h \underline{u}_h}{\underline{u}_h} &= \sup_{\underline w_h\in \mathbb R^{\dim Q_h}}\frac{\dual{K_h \underline{u}_h}{\underline{w}_h}^2}{\dual{M_h \underline{w}_h}{\underline{w}_h}} \\
&= \sup_{w_h\in Q_h}\frac{\inner{- \varepsilon \Delta u_h +\beta\cdot \nabla u_h + \sigma u_h }{w_h}^2_{L^2(\Omega)}}{\|w_h\|^2_{\LLO}}.
\end{align*}
Since $(-\varepsilon\Delta +\beta\cdot\nabla+ \sigma)U_h \subset Q_h$, the supremum is attained for $w_h = -\varepsilon\Delta u_h+\beta\cdot\nabla u_h+ \sigma u_h $, and we have
\begin{align*}
    \dual{K^T_h M^{-1}_h K_h \underline{u}_h}{\underline{u}_h} &=
    \sup_{w_h\in Q_h}\frac{\inner{- \varepsilon \Delta u_h +\beta\cdot \nabla u_h + \sigma u_h }{w_h}^2_{L^2(\Omega)}}{\|w_h\|^2_{\LLO}} \\&= \|- \varepsilon \Delta u_h +\beta\cdot \nabla u_h + \sigma u_h \|^2_{L^2(\Omega)}
    = \dual{B_h \underline{u}_h}{\underline{u}_h}.
\end{align*}
    Therefore,
    $K^T_h M^{-1}_h K_h=B_h$ and thus $\mathcal{S}(\mathcal{A}_h)=\mathcal{S}_h$.
\end{proof}

\section{Isogeometric analysis}
\label{sec:IGA}

Due to requirement $U_h \subset H^2(\Omega)\cap H^1_0(\Omega)$, we need a smooth discretization space. We achieve this by using IgA. We give a brief introduction to the approximation spaces in use. Let $S_{p,k,\ell}(0,1)$ be the space of B-spline functions on the unit interval $(0, 1)$ which are $k$-times continuously differentiable and piecewise polynomials of degree $p$ on a uniform grid with grid size $2^{-\ell}$.
For the space of B-splines with maximum continuity, that is, with $k = p-1$, we only write $S_{p,\ell}$.

On the parameter domain $\widehat{\Omega} := (0,1)^d$, we use a tensor-product B-spline space, denoted by
\[
   S^d_{p,k,\ell} :=\bigotimes^d_{i=1}S_{p,k,\ell}(0,1).
\]
For ease of notation, we assume to have the same spline degree $p$,
the same continuity $k$ and the same number of uniform refinement steps $\ell$, for each spatial dimension. We assume that the domain $\Omega$ can be parametrized by a geometry mapping $\mathbf{G}: \widehat{\Omega}
\rightarrow \Omega = G(\widehat{\Omega})$ with the property 
\begin{equation}
\label{eq:GeoMapCond}
    \|\nabla^r \mathbf{G}\|_{L^\infty(\widehat{\Omega})} \leq c_1 \quad \text{and} \quad \|\left(\nabla^r \mathbf{G}\right)^{-1}\|_{L^\infty(\widehat{\Omega})} \leq c_2, \quad  \text{for}\quad r=1,2,3,
\end{equation}
for some constants $c_1$ and $c_2$. The discretization space $S_{p,k,\ell}$ on the domain $\Omega$ is defined using the pull-back principle as 
\begin{equation*}
S_{p,k,\ell} (\Omega) := \left\lbrace f \circ \mathbf{G}^{-1} : f \in S^d_{p,k,\ell}\right\rbrace.
\end{equation*}
For more information on IgA, see the survey article \cite{VeiBufSanVaz14} and the references therein. We use spline spaces with maximum smoothness as the discrete state space and reduce the smoothness for $Q_h$ accordingly. More precisely, we use 
\begin{equation}
\label{eq:stateSpaceDisc}
    Q_h := S_{p,p-3,\ell}(\Omega) 
    \quad \text{and} \quad
    U_h := S_{p,\ell}(\Omega) \cap H^1_0(\Omega)
    \quad \text{with} \quad p\geq 2.
\end{equation}
The following lemma shows that, if we consider the special case of box domains (if trivially parametrized) and constant convection, the condition of Lemma~\ref{lemma:conformingDisc} holds.
\begin{theorem}
\label{theorm:conformingSplines}
If $\Omega=(0,1)^d$,
 $Q_h =  S_{p,p-3,\ell}^d$ and $U_h=S_{p,\ell}^d\cap H^1_0(\Omega)$ and if the convection $\beta$ is constant, then
 \[
    \mathcal{S}(\mathcal{A}_h)  =  \mathcal{S}_h \quad \text{and}\quad  \kappa\left(\mathcal{S}^{-1}_h\mathcal{A}_h\right) \leq  \frac{\cos (\pi/7)}{\sin (\pi/14)}
    \approx 4.05.
 \]
 \end{theorem}
 \begin{proof}
 For sake of simplicity, we restrict the proof to the two-dimensional case. Clearly,
 \[
  f'  \in  S_{p-1,k-1,\ell}(0,1)\quad \foralls f\in S_{p,k,\ell}(0,1)
  \]
  together with
  \[S_{p-1,k-1,\ell}(0,1)\subset S_{p,k-1,\ell}(0,1)\quad \text{and}\quad S_{p,k,\ell}(0,1)\subset S_{p,k-2,\ell}(0,1),
  \]

    see~\cite{VeiBufSanVaz14}.
  Let $u\in S^2_{p,k,\ell}= S_{p,k,\ell}(0,1)\otimes S_{p,k,\ell}(0,1)$. Then,
  \begin{align*}
    \frac{\partial^2 u}{\partial x_1^2} &\in S_{p-2,k-2,\ell}(0,1)\otimes S_{p,k,\ell}(0,1)\subset S^2_{p,k-2,\ell},
    \\
    \frac{\partial^2 u}{\partial x_2^2} &\in S_{p,k,\ell}(0,1)\otimes S_{p-2,k-2,\ell}(0,1)\subset S^2_{p,k-2,\ell},
  \end{align*}
and, by combining these results, also
\[\Delta u = \frac{\partial^2 u}{\partial x_1^2} + \frac{\partial^2 u}{\partial x_2^2} \in S^2_{p,k-2,\ell}.\]
Since $\beta= (\beta_1,\beta_2)$ is constant we also have
\[
\beta\cdot\nabla u = \beta_1\frac{\partial u}{\partial x_1} + \beta_2\frac{\partial u}{\partial x_2} \in S^2_{p,k-2,\ell}.
\]
To summarize, for every $u\in S^2_{p,k,\ell}$, we have
\[
-\varepsilon \Delta u + \beta \cdot \nabla u + \sigma u \in S^2_{p,k-2,\ell},
\]
hence Condition~(\ref{eq:conformingDisc}) in Lemma~\ref{lemma:conformingDisc} holds and we have $\mathcal{S}(\mathcal{A}_h)  =  \mathcal{S}_h$. The condition number bound follows from Theorem~\ref{theo:schurPos}.
\end{proof}

If the domain is not a box domain or if the convection is not constant, then Theorem~\ref{theorm:conformingSplines} cannot be applied.
Nevertheless, the numerical results presented in Section~\ref{sec:num} indicate that the spaces proposed in (\ref{eq:stateSpaceDisc}) work well even for more complex domains and variable convection. 

\section{Error estimates}
\label{sec:errorEst}
In this section, we derive discretization error estimates for Problem~\ref{prob:disc}. Let $A(\mathbf{x},\tilde{\mathbf{x}})$ denote the bilinear form in Problem \ref{prob:KKTCont}, where $\mathbf{x}, \tilde{\mathbf{x}} \in \mathbf{X}$ denotes the triplet $(q,w,u)$ $\in L^2(\Omega) \times L^2(\Omega) \times H^2(\Omega)\cap H^1_0(\Omega)$, with the norm
\[
\| \mathbf{x} \|^2_{\mathcal{S}} = \|q \|^2_{S_q} + \|w \|^2_{S_w} + \|u \|^2_{S_u}. 
\]The discrete triplet $(q_h,w_h,u_h) \in Q_h \times Q_h \times U_h$ is denoted by $\mathbf{x}_h \in \mathbf{X}_h$. Galerkin orthogonality reads as follows:
\begin{equation}
\label{eq:GalerkinOth}
A(\mathbf{x}-\mathbf{x}_h,\mathbf{y}_h) = 0 \quad \foralls \mathbf{y}_h \in \mathbf{X}_h.
\end{equation}
Since Problem~\ref{prob:KKTCont} is well-posed (Theorem~\ref{theo:schurPos}), we have boundedness
\begin{equation}
\label{eq:boundedness}
A(\mathbf{x},\mathbf{y}) \leq \overline{c}\|\mathbf{x}\|_{\mathcal{S}} \|\mathbf{y}\|_{\mathcal{S}} \quad \foralls \mathbf{x},\mathbf{y}\in \mathbf{X}
\end{equation}
and
inf-sup stability
\begin{equation}
\label{eq:coercivity}
\sup_{0 \neq \mathbf{y} \in \mathbf{X}} \frac{A(\mathbf{x},\mathbf{y})}{\|\mathbf{y}\|_{\mathcal{S}}} \geq \underline{c}\|\mathbf{x}\|_{\mathcal{S}} \quad \foralls \mathbf{x}\in \mathbf{X},
\end{equation}
where $\overline c/\underline c =  \kappa(\mathcal S^{-1}\mathcal A)$.
The boundedness also holds for a conforming discretization space $\mathbf{X}_h \subset \mathbf{X}$ with the same constant $\overline{c}$. If the condition (\ref{eq:conformingDisc}) in Lemma~\ref{lemma:conformingDisc} holds, then also the inf-sup holds with the same constant $\underline{c}$. Using this and the ideas of \cite{babuvska1971error}, we derive the
following discretization error estimate.
\begin{lemma}
\label{lemma:Cea}
If $\mathbf{x}_h \in \mathbf{X}_h$ is the solution to (\ref{eq:MatDis1}) for a discretization space satisfying condition (\ref{eq:conformingDisc}) and if $\mathbf{x}\in \mathbf{X}$ is the solution of (\ref{eq:MatCon1}), then we have the estimate
\begin{equation}
\label{eq:Cea}
    \|\mathbf{x}- \mathbf{x}_h\|_{\mathcal{S}} \leq (1+\kappa(\mathcal S^{-1}\mathcal A)) \inf_{\mathbf{y}_h\in \mathbf{X}_h}\|\mathbf{x}- \mathbf{y}_h\|_{\mathcal{S}}.
\end{equation}
\end{lemma}
\begin{proof}
Let $\mathbf{y}_h$ and $\mathbf{z}_h$ be arbitrary functions in $\mathbf{X}_h$. Due to Galerkin orthogonality (\ref{eq:GalerkinOth}), we have
\begin{align*}
    A(\mathbf{x}_h-\mathbf{y}_h,\mathbf{z}_h)  =  A(\mathbf{x}_h-\mathbf{x}+\mathbf{x}-\mathbf{y}_h,\mathbf{z}_h) = A(\mathbf{x}-\mathbf{y}_h,\mathbf{z}_h).
\end{align*}
Combining this with the boundedness condition (\ref{eq:boundedness}) gives
\begin{equation*}
A(\mathbf{x}_h-\mathbf{y}_h,\mathbf{z}_h)  \leq \overline{c} \|\mathbf{x}-\mathbf{y}_h\|_{\mathcal{S}} \|\mathbf{z}_h\|_{\mathcal{S}} \quad \foralls \mathbf{y}_h,\mathbf{z}_h\in \mathbf{X}_h.
\end{equation*}
Using the discrete inf-sup gives
\begin{equation*}
\underline{c}\|\mathbf{x}_h-\mathbf{y}_h\|_{\mathcal{S}}
\leq \sup_{0 \neq \mathbf{z}_h \in \mathbf{X}_h} \frac{A(\mathbf{x}_h-\mathbf{y}_h,\mathbf{z}_h)}{\|\mathbf{z}_h\|_{\mathcal{S}}} \leq \overline{c} \|\mathbf{x}-\mathbf{y}_h\|_{\mathcal{S}} \quad \foralls \mathbf{y}_h\in \mathbf{X}_h.
\end{equation*}
Finally, we use the triangle inequality and $\overline c/\underline c =\kappa(\mathcal S^{-1}\mathcal A)$ to obtain the desired result 
\[
\|\mathbf{x}-\mathbf{x}_h\|_{\mathcal{S}} \leq \|\mathbf{x}-\mathbf{y}_h\|_{\mathcal{S}} +\|\mathbf{x}_h-\mathbf{y}_h\|_{\mathcal{S}} \leq (1+\kappa(\mathcal S^{-1}\mathcal A)) \|\mathbf{x}-\mathbf{y}_h\|_{\mathcal{S}} \quad \foralls \mathbf{y}_h\in \mathbf{X}_h.
\]
\end{proof}
Let $\Omega = \widehat{\Omega} := (0,1)^d$
and let the convection $\beta$ be constant. Since we assume
a trivial parametrization, we have
the discretization spaces $Q_h = S^d_{p,p-3,\ell}$ and $U_h = S^d_{p,\ell}\cap H^1_0(\widehat\Omega)$. 

To derive the error estimates, we assume that the solution of (\ref{eq:MatCon1}) satisfies the regularity assumption
$(q,w,u)\in H^1(\widehat{\Omega})\times H^1(\widehat{\Omega})\times H^3(\widehat{\Omega}) \cap H^1_0(\widehat{\Omega})$.

Now, we estimate the approximation error term in (\ref{eq:Cea}) from above. First, we observe that
\begin{align*}
\inf_{\mathbf{y}_h\in \mathbf{X}_h}\|\mathbf{x}- \mathbf{y}_h\|_{\mathcal{S}}
\quad &\le \\ \inf_{q_h\in {S}^d_{p,p-3,\ell}} \alpha\|q-q_h\|^2_{L^2({\widehat{\Omega}})}
&+ \inf_{w_h\in {S}^d_{p,p-3,\ell}} \frac{1}{\alpha}\|w-w_h\|^2_{L^2({\widehat{\Omega}})} +
\inf_{u_h\in {S}^d_{p,\ell}\cap H^1_0(\widehat\Omega)} \|u-u_h\|^2_{S_u}.
\end{align*}
The two first terms can be bounded by using the following approximation error estimate
\begin{equation}
\label{eq:AppNonSmooth}
\inf_{q_h\in S^d_{p,p-3,\ell}} \|q-q_h\|_{L^2(\widehat\Omega)} \leq \frac{h}{4\sqrt{3}} \|\nabla q\|_{L^2(\widehat\Omega)},
\end{equation}
see \cite[Corollary 1]{sande2020explicit}. The estimate for the last term 
\[
\inf_{u_h\in S^d_{p,\ell}\cap H^1_0(\widehat\Omega)} \|u-u_h\|^2_{S_u}
\]
is more involved. Before handling this term, we need a convenient
notation and some auxiliary approximation error estimates.
\begin{notation}
In what follows, $c$ is a generic positive constant independent of $\alpha$, $\sigma$, $\beta$, $\varepsilon$, $h$ and $p$, 
but may depend on the spatial dimension $d$ and the observation domain $\mathcal{O}$. 
\end{notation}
 In Appendix~\ref{sec:app2}, we extend some of the results of \cite{sande2020explicit,sogn2018schur,sogn2019robust}. The main result is summarized in the following theorem.
 
\begin{theorem}
\label{theo:App210}
Let $\mathbf{\Pi}_p:H^3(\widehat{\Omega})\cap H^1_0(\widehat{\Omega})\rightarrow S^d_{p,\ell} \cap H^1_0(\widehat{\Omega})$ be the $H^2$-orthogonal projector, where $p\geq 3$ and $\ell \geq 1$. Then, 
\begin{align}
    \label{eq:AppH2H3}
    \|\nabla^2 (I-\mathbf{\Pi}_p)u \|_{L^2(\widehat{\Omega})} &\leq c h^{\:\:}  \|\nabla^3 u \|_{L^2(\widehat{\Omega})},\\
    \label{eq:AppH1H3}
    \|\nabla (I-\mathbf{\Pi}_p)u \|_ {L^2(\widehat{\Omega})} &\leq c h^2  \|\nabla^3 u \|_{L^2(\widehat{\Omega})},\\
    \label{eq:AppL2H3}  
    \|(I-\mathbf{\Pi}_p)u\|_{L^2(\widehat{\Omega})} &\leq c h^3  \|\nabla^3 u \|_{L^2(\widehat{\Omega})} \quad \foralls u \in H^3(\widehat{\Omega})\cap H^1_0(\widehat{\Omega}).
\end{align}
\end{theorem}
\begin{remark}
In the statement of
Theorem~\ref{theo:App210}, the parameter domain $\widehat\Omega$ is considered. This result can be extended to physical domains if the corresponding geometry function $\mathbf{G}$ is sufficiently smooth, cf. \cite{sogn2019robust}. 
\end{remark}
With Theorem~\ref{theo:App210}, we can derive an error estimate for our problem. 
\begin{theorem}
\label{theo:errorEst}
Let $\Omega := \widehat\Omega := (0,1)^d$ with $d\in \mathbb{N}$, let $\beta\in \mathbb{R}^d$ be constant and let $(q_h,w_h,u_h) \in S^d_{p,p-3,\ell} \times S^d_{p,p-3,\ell} \times S^d_{p,\ell}\cap H^1_0(\widehat{\Omega})$ with $p\geq 3$ and $\ell \geq 1$, be the solution to (\ref{eq:MatDis1}). If $(q,w,u) \in H^1(\widehat\Omega) \times H^1(\widehat\Omega) \times H^3(\widehat\Omega) \cap H^1_0(\widehat\Omega)$ is the solution of (\ref{eq:MatCon1}), then we have the following estimates:
\begin{align*}
    &\|q - q_h\|_{S_q} + \|w - w_h\|_{S_w} +\|u - u_h\|_{S_u} \leq \\
    &c h \!\left(\! \sqrt{\alpha} \|\nabla q\|_{L^2(\widehat\Omega)} +\frac{1}{\sqrt{\alpha}} \|\nabla w\|_{L^2(\widehat\Omega)} + \sqrt{\alpha}\, \max\left\lbrace \varepsilon, \|\beta\| h, \left(\sigma \!+\! \frac{1}{\sqrt{\alpha}}\right)\! h^2 \!\right\rbrace  \|\nabla^3 u\|_{L^2(\widehat\Omega)}\right)
\end{align*}
\end{theorem}
\begin{proof}
Let $\tilde{u}_h := \mathbf{\Pi}_p u$. Then, we have for all $u \in H^3(\widehat{\Omega})\cap H^1_0(\widehat{\Omega})$
\begin{align*}
    \|u-\tilde{u}_h\|^2_{S_u} &= \|u - \tilde{u}_h \|^2_{L^2(\mathcal{O})} +\alpha \|(-\varepsilon \Delta + \beta \cdot \nabla + \sigma ) (u - \tilde{u}_h) \|^2_{L^2(\widehat{\Omega})}.
    \end{align*}
For the first term, we simply extend the norm from $\mathcal{O}$ to $\widehat{\Omega}$ and obtain
\[
\|u - \tilde{u}_h \|^2_{L^2(\mathcal{O})} \leq \|u - \tilde{u}_h \|^2_{L^2(\widehat{\Omega})} \leq c\, h^6 \|\nabla^3 u\|^2_{L^2(\widehat{\Omega})}. 
\]
For the second term, we use the triangle inequality and Theorem~\ref{theo:App210}, and we get 
    \begin{align*}
    &\alpha \|(-\varepsilon \Delta + \beta \cdot \nabla + \sigma ) (u - \tilde{u}_h) \|^2_{L^2(\widehat{\Omega})}\\
&\quad\leq3 \alpha \left(\varepsilon^2 \|\Delta (u-\tilde{u}_h) \|^2_{L^2(\widehat{\Omega})} + \|\beta\|^2\|\nabla(u-\tilde{u}_h) \|^2_{L^2(\widehat{\Omega})}  + \sigma^2 \|u-\tilde{u}_h \|^2_{L^2(\widehat{\Omega})} \right)\\
&\quad\leq c\, \alpha \left( \varepsilon^2 h^2 + \|\beta\|^2 h^4 + \sigma^2 h^6 \right)\|\nabla^3 u\|^2_{L^2(\widehat{\Omega})}.
\end{align*}
Combining these results gives
\begin{equation}
\label{eq:AppU}
\inf_{\tilde{u}_h\in S_{p,\ell}^d\cap H^1_0(\widehat\Omega)} \|u-\tilde{u}_h\|_{S_u}\leq  c\, \sqrt{\alpha}\, h\,\max\left\lbrace \varepsilon, \|\beta\| h, (\sigma+ 1/\sqrt{\alpha}) h^2 \right\rbrace  \|\nabla^3 u\|_{L^2(\widehat{\Omega})}.
\end{equation}
The theorem follows from combining Lemma~\ref{lemma:Cea} with Equation~(\ref{eq:AppNonSmooth}) and Equation~(\ref{eq:AppU}).
\end{proof}
\begin{remark}
Theorem~\ref{theo:errorEst} is somewhat restrictive since it requires the domain to be a box domain and the convection to be constant. These requirements are needed for the proof of the discrete inf-sup stability. 
The numerical results presented in Section~\ref{sec:num} indicate that the restrictions are not needed. 
\end{remark}

\section{Numerical experiments: Accuracy of the solution}
\label{sec:NumCanonical}
In this section we compare the solution of the forward
problem to the (state) solution of the optimal control problem
to investigate the fact that 
the optimal control problem naturally introduces
a non-standard Petrov-Galerkin method for the state equation. 
We consider a well-known one-dimensional problem~\cite{brooks1982streamline}:
\[
-\partial_x u(x) - \varepsilon \,\partial_{xx} u(x) = 0 \quad \mbox{in}\quad (0,1),\qquad 
u(0) = 0,\qquad
u(1) = 1,
\]
whose exact solution is
\[
u(x) = \frac{e^{-x/\varepsilon} - 1}{e^{-1/\varepsilon} - 1}. 
\]
This problem is used as state equation in  our optimal control problem (Problem~\ref{prob:main}),
where we choose $\beta = -1$, $\sigma = 0$, $f=0$, $u_d$ to be the exact solution of
the forward problem, and the boundary conditions on $u$ to be as for the forward problem.
The analytical solution of the optimal control problem is 
\[
q(x) = 0, \quad w(x)= 0, \quad u(x) = \frac{e^{-x/\varepsilon} - 1}{e^{-1/\varepsilon} - 1}. 
\]
We use the discretization spaces 
\[
Q_h = S_{p,p-3,\ell}(0,1) \quad \text{and} \quad  U_h = \left\lbrace u_h\in S_{p,\ell}(0,1)\, : \, u(0) = 0,\, u(1)= 1 \right\rbrace
\]
for the optimal control problem, which satisfy the
condition (\ref{eq:conformingDisc}). We compare the numerical solution for the state with the numerical solution of the forward problem, where we use $U_h$ as trial and test space. No stabilization techniques are used. The diffusion is set to $\varepsilon = 0.01$ and $\alpha = 0.001$ for the optimal control problem. Three observation domains are considered: Full observation, that is, $\mathcal{O}=\Omega=(0,1)$, and partial observation on $\mathcal{O}=(0,\frac{1}{4})$ and on $\mathcal{O}=(\frac{3}{4},1)$. The numerical solutions are displayed in Figures~\ref{plot:Full} to \ref{plot:EndRefine}. The plots indicate that the forward solution is unstable for coarse discretizations. The non-physical oscillations start in the boundary layer and propagate into the remainder of the computational domain. These kinds of instabilities are often remedied by upwind and/or Petrov-Galerkin schemes~\cite{brooks1982streamline}. We remark though that our Petrov-Galerkin like approach for the state equation does not resemble any of the common Petrov-Galerkin schemes for this equation, as far as we know. The state solution (of the optimal control problem) does not have these instabilities.  In fact, the state operator $K_h$ is discretized with a Petrov–Galerkin method as the trial space is $U_h$ and the test space is $Q_h$.  
\begin{figure}[h]
  \center
  \includegraphics[width=0.45\textwidth]{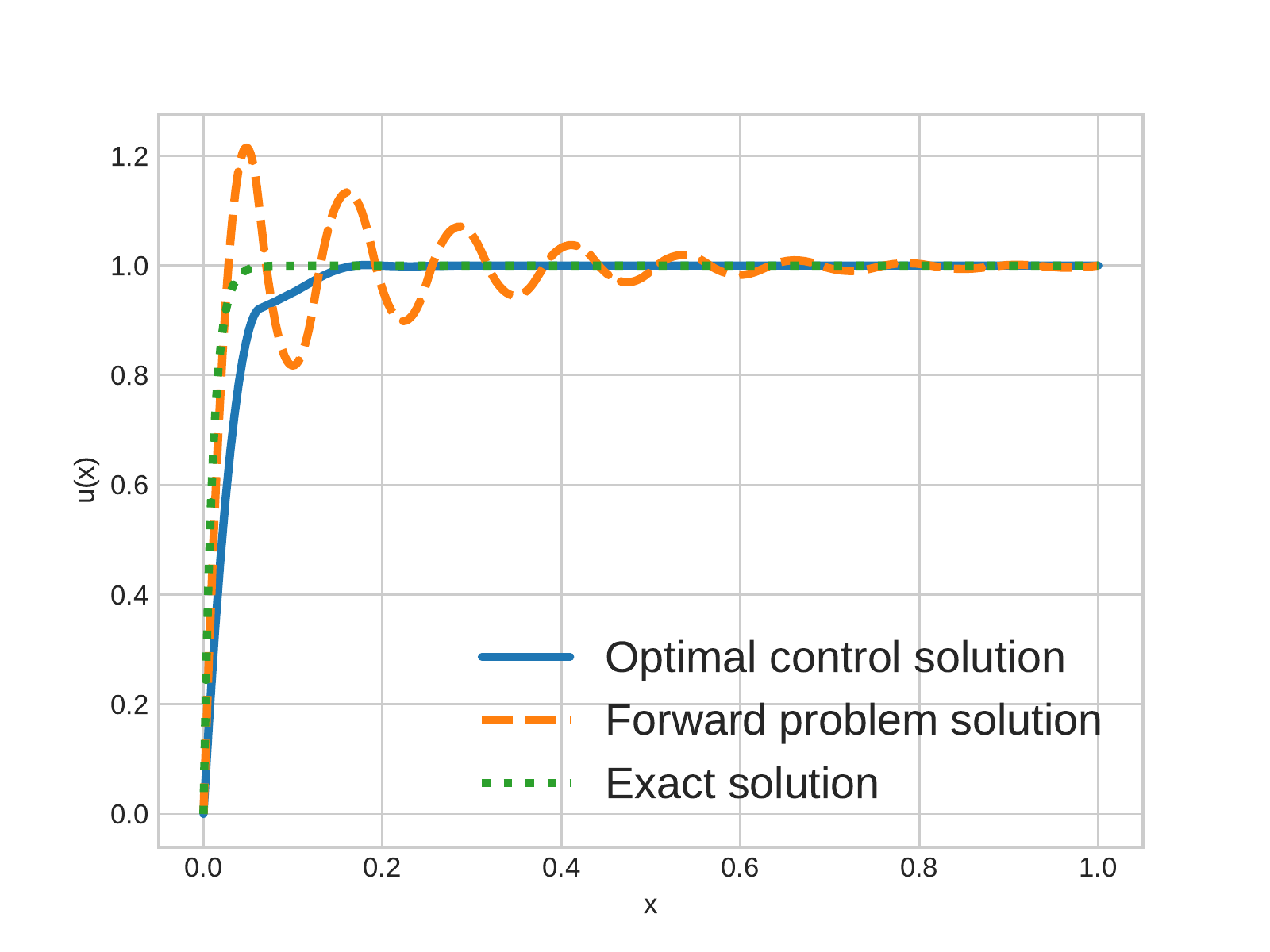}
  \includegraphics[width=0.45\textwidth]{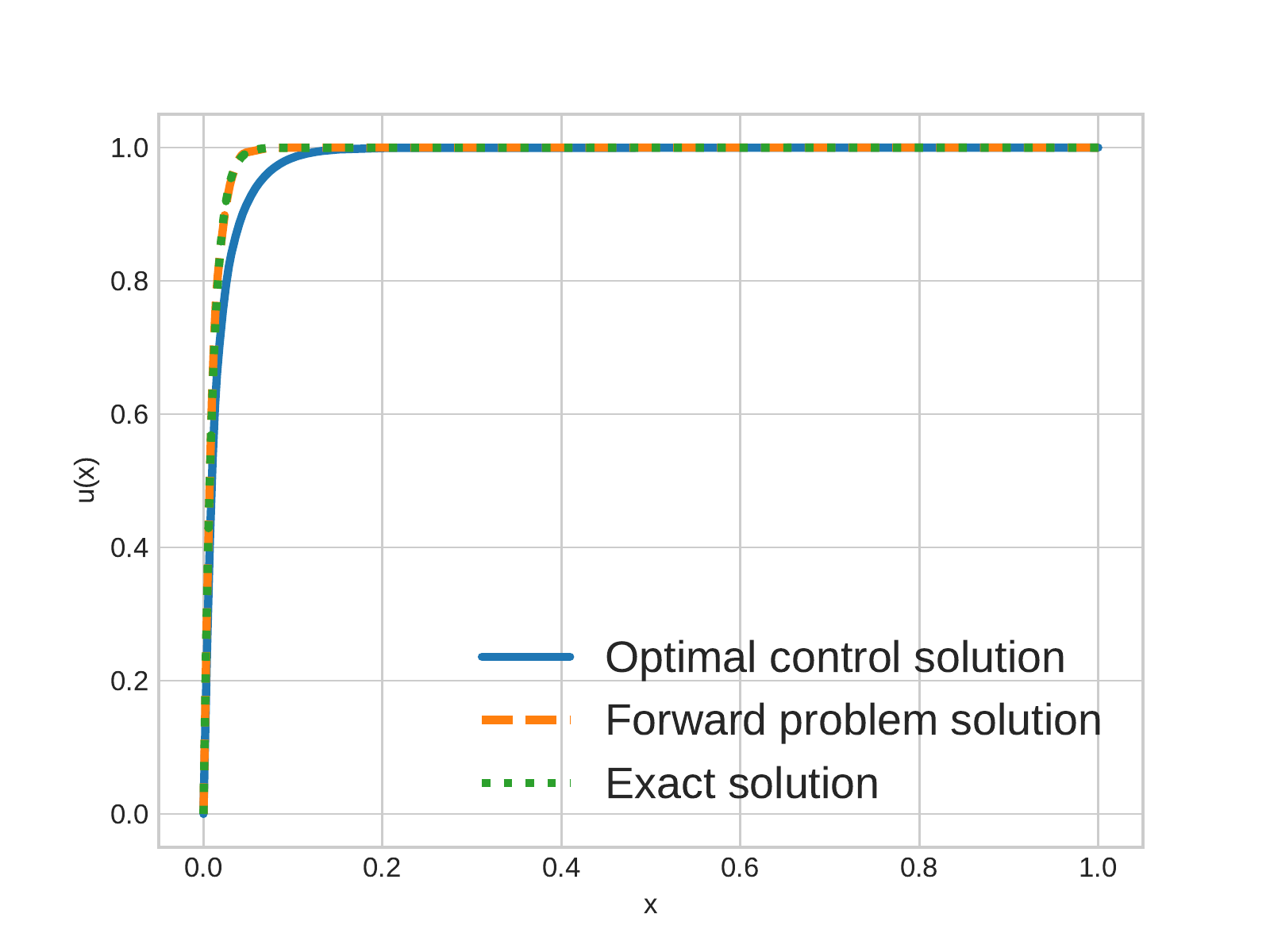}
  \caption{Full observation on $\mathcal{O} = (0,1)$ with $p = 2$ (both), $\ell = 4$ (left), $\ell = 6$ (right).}
  \label{plot:Full}
\end{figure}
\begin{figure}[h]
  \center
  \includegraphics[width=0.45\textwidth]{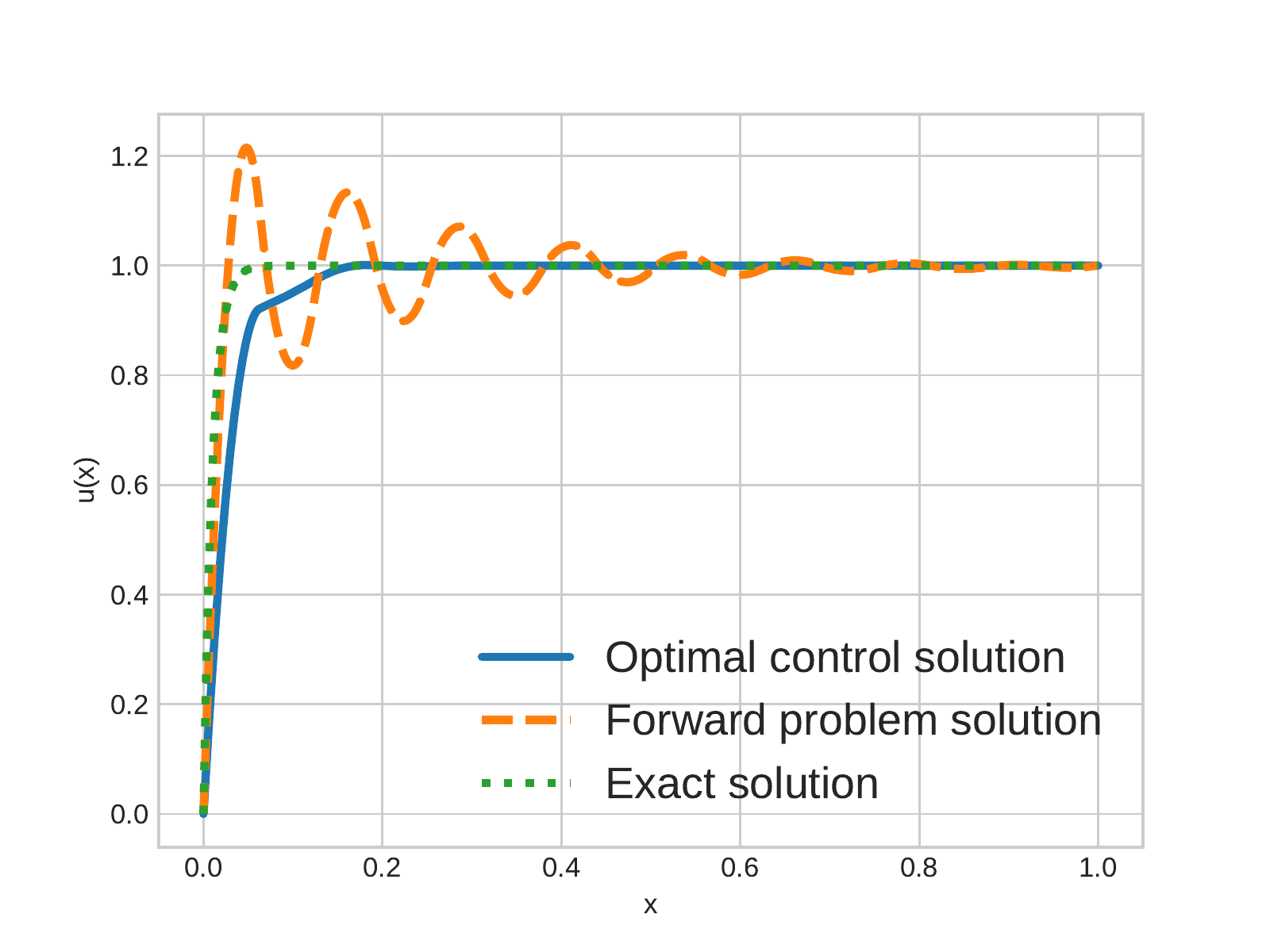}
  \includegraphics[width=0.45\textwidth]{tr6front.pdf}
  \caption{Partial observation on $\mathcal{O} = (0,\frac{1}{4})$ with $p = 2$ (both), $\ell = 4$ (left), $\ell = 6$ (right).}
  \label{plot:Front}
\end{figure}
\begin{figure}[h]
  \center
  \includegraphics[width=0.45\textwidth]{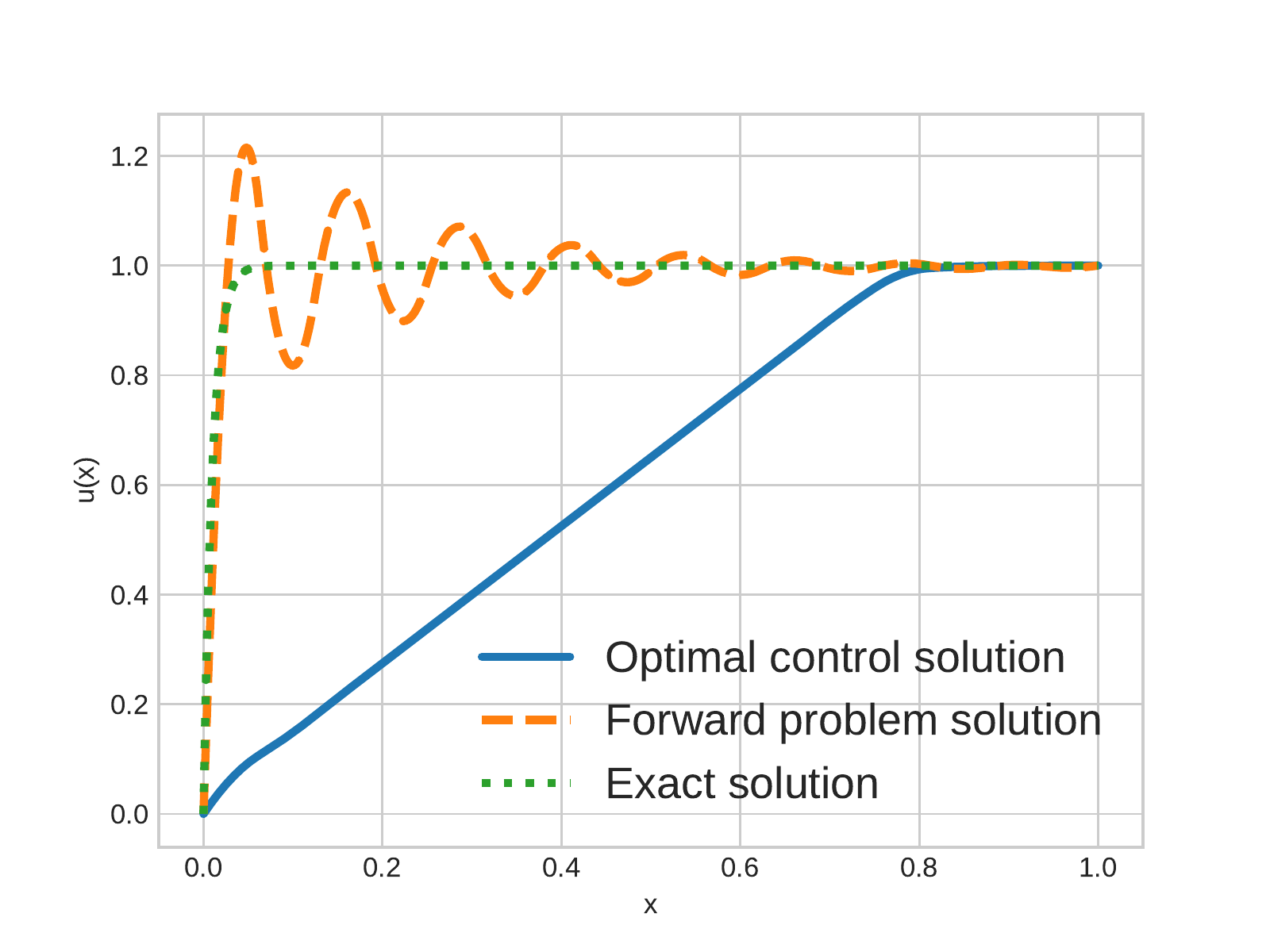}
  \includegraphics[width=0.45\textwidth]{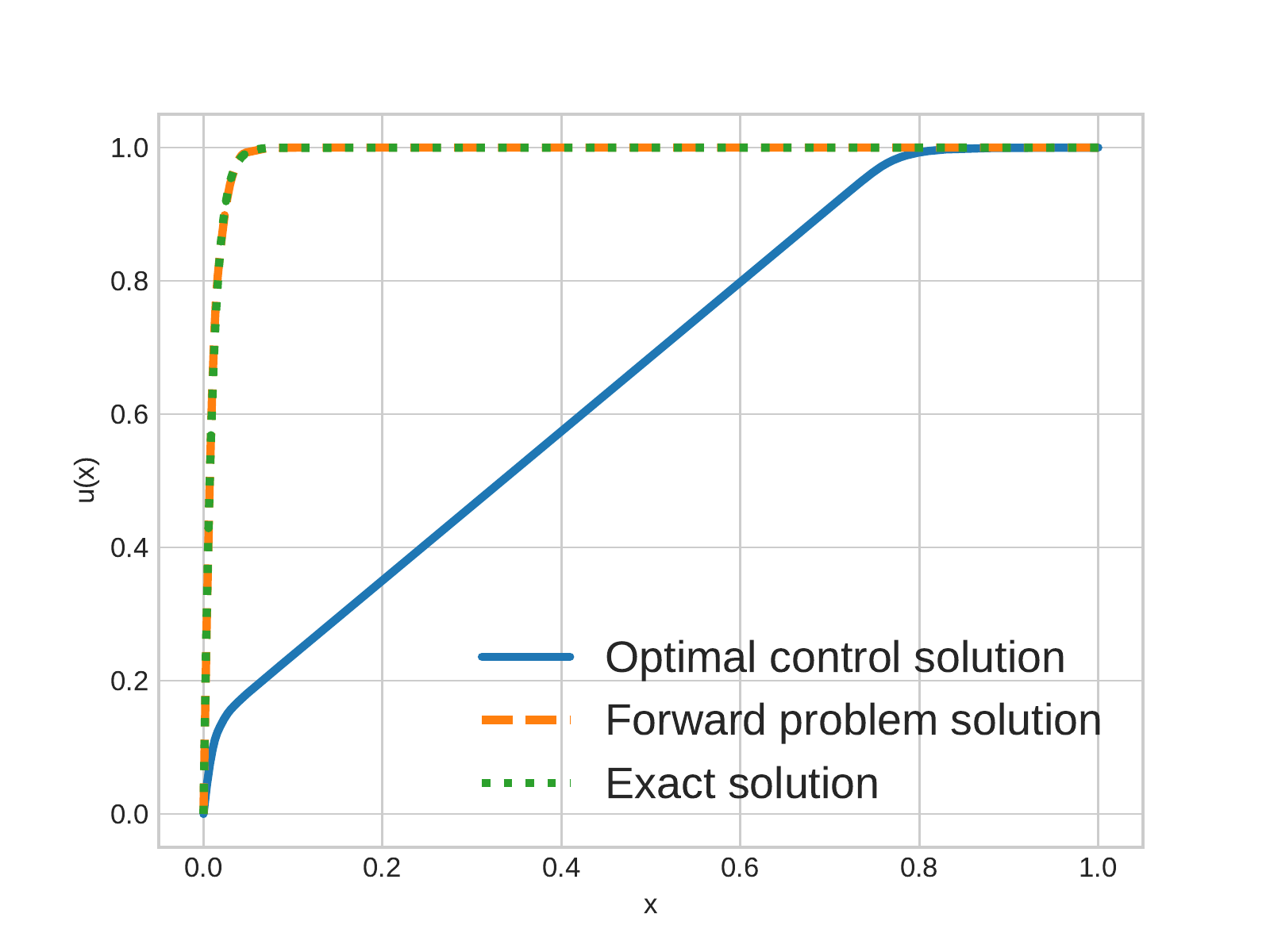}
  \caption{Partial observation on $\mathcal{O} = (\frac{3}{4},1)$ with $p = 2$ (both), $\ell = 4$ (left), $\ell = 6$ (right).}
  \label{plot:End}
\end{figure}

\begin{figure}[h]
  \center
  \includegraphics[width=0.45\textwidth]{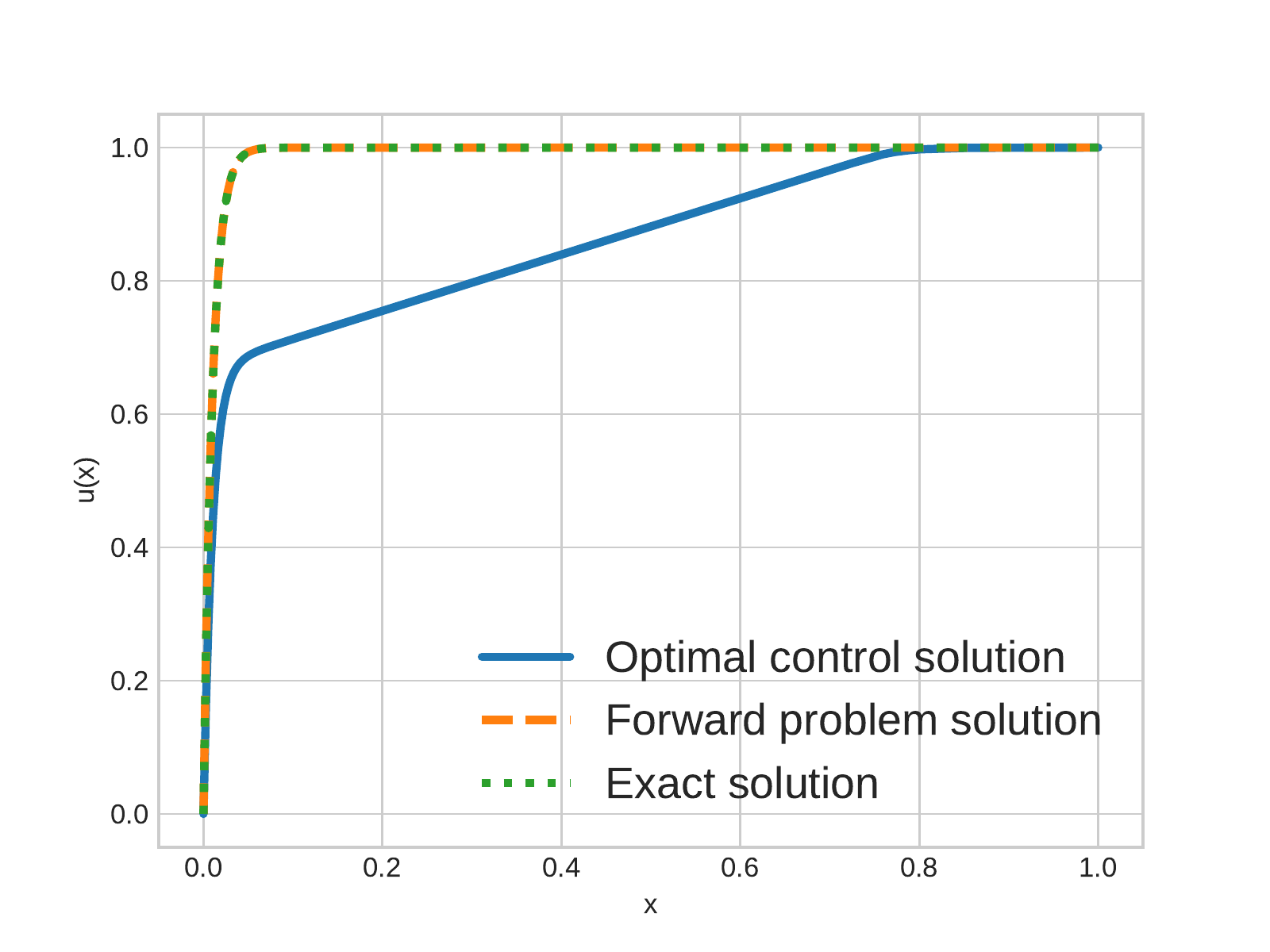}
  \includegraphics[width=0.45\textwidth]{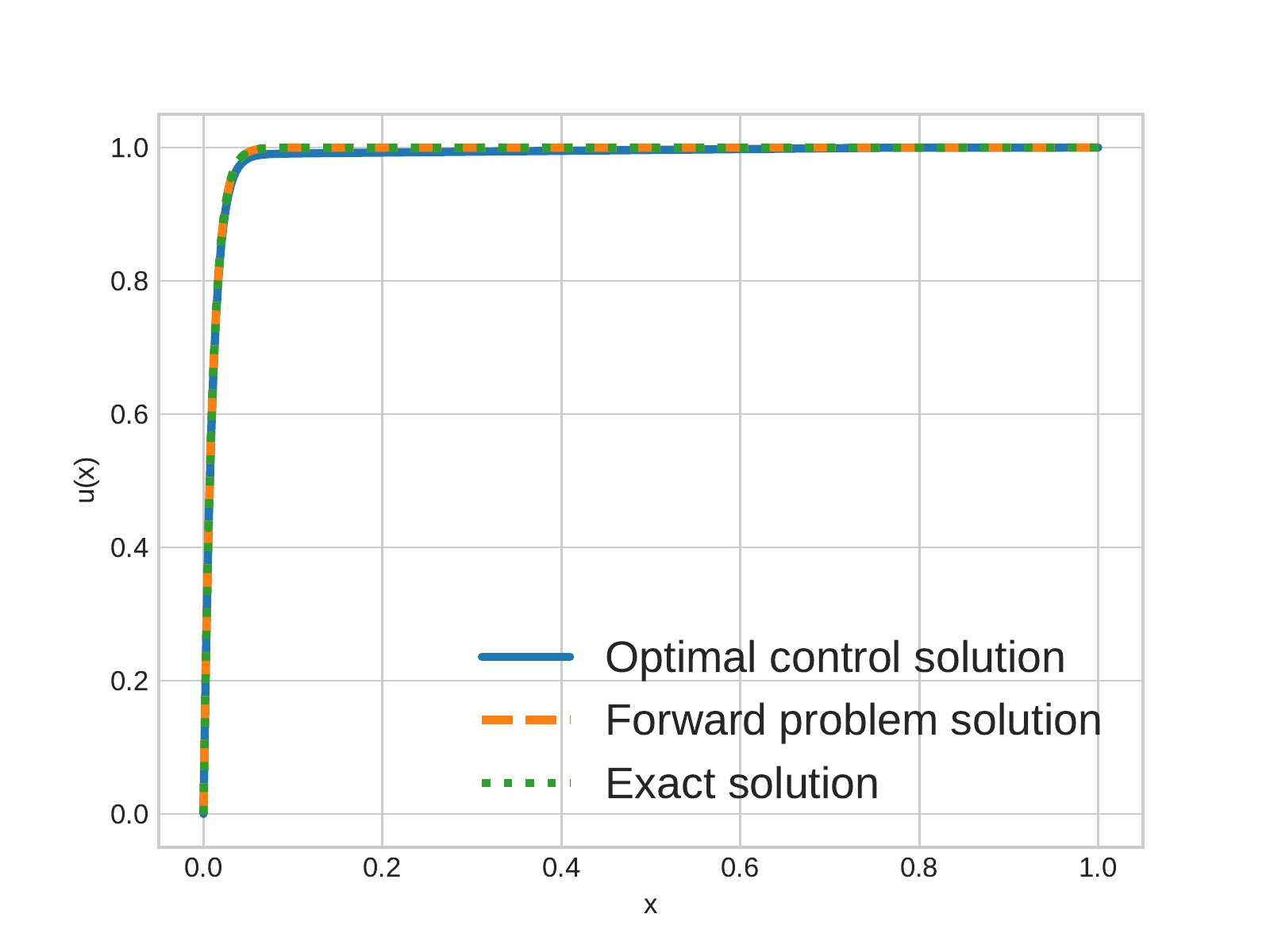}
  \caption{Partial observation on $\mathcal{O} = (\frac{3}{4},1)$ with $p = 2$, $\ell = 8$ (left) and  $p = 4$, $\ell = 6$ (right).}
  \label{plot:EndRefine}
\end{figure}

In Figure~\ref{plot:Front}, we consider the optimal control problem with observation on $(0,\frac{1}{4})$. This is only a quarter of the whole domain, but it is located at the boundary layer. The solutions for
the state variable are almost identical to those obtained for the case of full observation.

Next we look at the solution where the observation domain is $(\frac{3}{4},1)$. Here the solution is almost constant ($u(x) \approx 1$). From Figures~\ref{plot:End} and \ref{plot:EndRefine}, we see that the approximation is not good. In the left plot of Figure~\ref{plot:End}, we see that the boundary layer is not captured. However, the error does not propagate into the observation domain. For $h$-refinement, see Figure~\ref{plot:End} (right) and Figure~\ref{plot:EndRefine} (left), we observe that the approximation improves slowly.
For $p$-refinement, see Figure~\ref{plot:EndRefine} (right), the approximation improves
significantly. Since we use splines, increasing the spline degree by one means that the
number of degrees of freedom is only increased by one, while each $h$-refinement step
doubles the number of degrees of freedom.

\begin{remark}
The effect of increasing the spline degree compared to $h$-refinement as shown in Figure~\ref{plot:EndRefine} is somewhat surprising. We are not completely sure why larger spline degrees are so effective. Unfortunately, the error estimate in Theorem~\ref{theo:errorEst} does not provide any explanation for this behavior. Further analysis is needed to explain this properly. 
\end{remark}

\section{Numerical experiments for exact and inexact preconditioners}
\label{sec:num}

In this section, we analyze the convergence of Krylov space solvers
when the proposed preconditioner is used. In the first subsection, we consider an exact realization of the preconditioner. A multigrid approximation is then considered in the second subsection.

\begin{figure}[h]
    \center
    \includegraphics[width=0.3\textwidth]{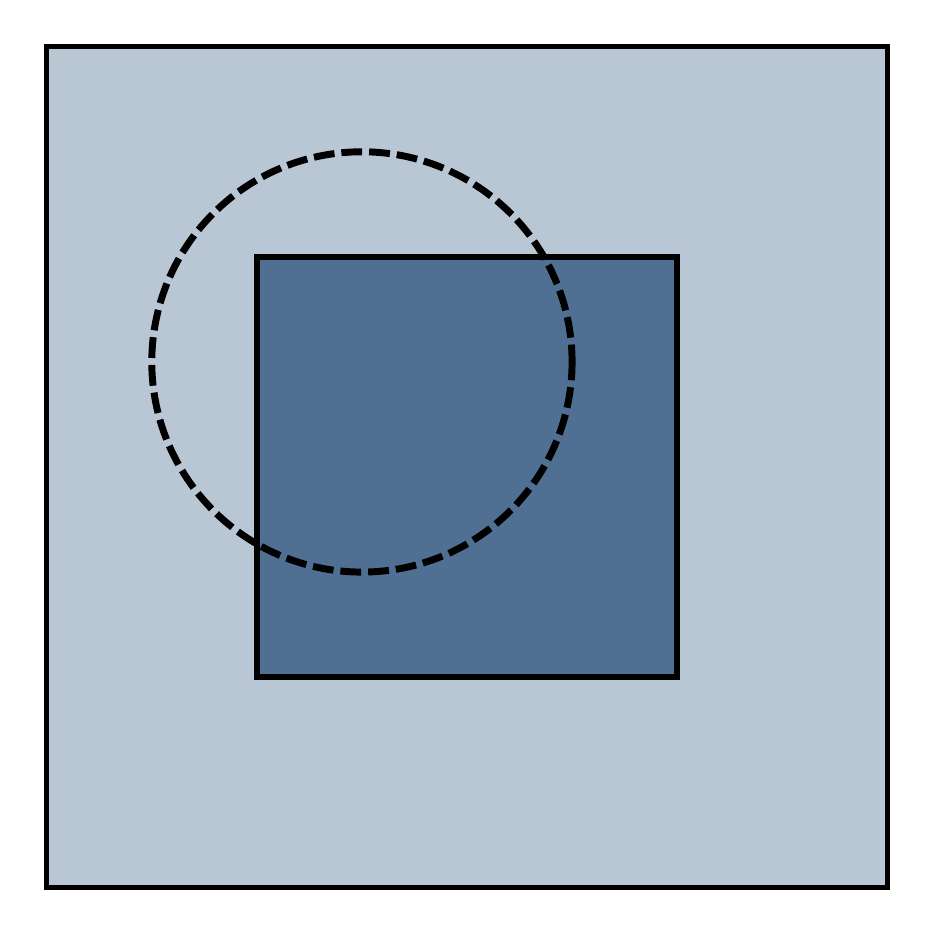}\quad
    \includegraphics[width=0.3\textwidth]{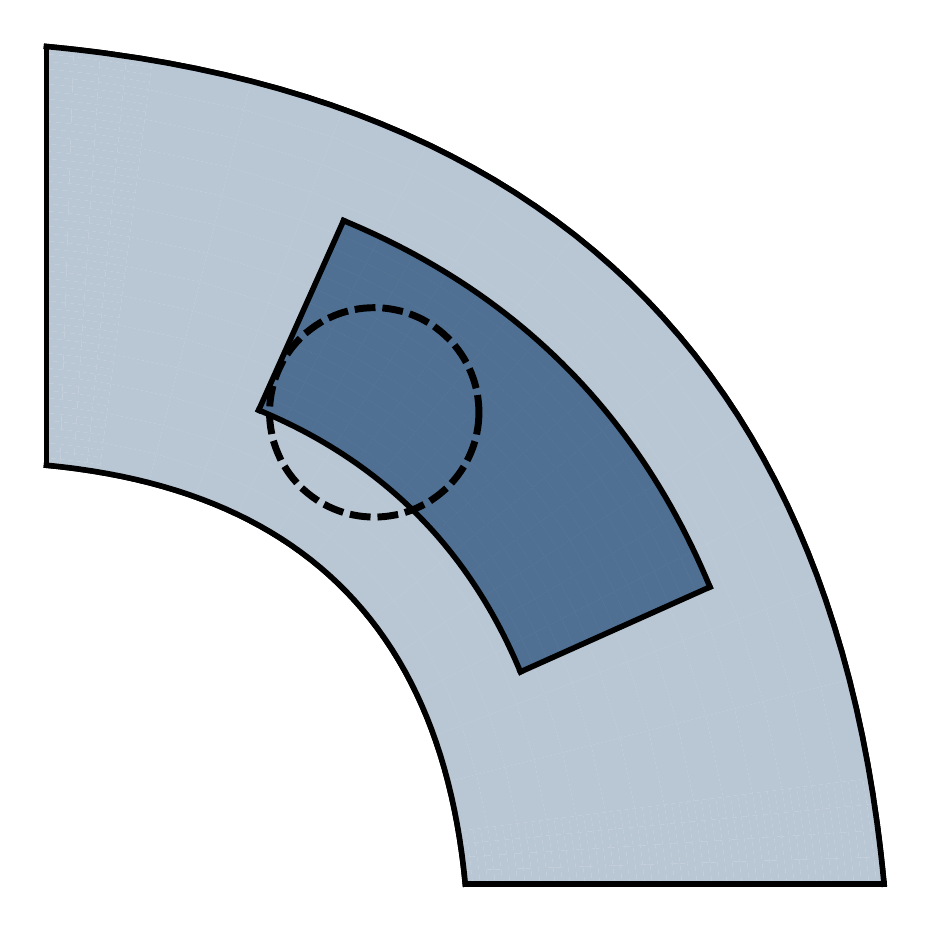}
    \caption{Computational domains $\Omega$, partial observation domains $\mathcal{O}$ (dark blue), and support of desired state (inside of dashed circles).}
    \label{fig:2DDomain}
\end{figure}

We have done the numerical experiments for two model domains, in both cases for $d=2$.
The first domain is a box-domain, more precisely, $\Omega$ is the unit square, see Figure~\ref{fig:2DDomain} (left), which is parameterized with the identity function.
For this domain, the conditions of Theorem~\ref{theorm:conformingSplines} are satisfied.
In Model problem~\ref{mp:1}, we have full observation and in Model problem~\ref{mp:1a}, the observation is restricted to the subdomain represented by the smaller area in Figure~\ref{fig:2DDomain} (left).
In all model problems, the desired state $u_d$ is a step function with value $u_d = 1$
inside a circle and with value $u_d = 0$ outside the circle. The support of $u_d$ is shown as the dashed lines in Figure~\ref{fig:2DDomain}. 

\begin{modelproblem}[Unit square and constant convection with full observation]
    \label{mp:1}
    Let $\Omega := \mathcal{O} := (0,1)^2$ be the computational domain, which is also the observation domain. The convection is $\beta = (-2,1)$ and there is no reaction $\sigma = 0$ or source term $f = 0$. The desired state is  
\[
    u_d(x,y)= 
\begin{cases}
    1 \quad &\text{if}\quad (x-\frac38)^2 + (y-\frac58)^2 \leq \frac{1}{16}\\
    0 \quad    & \text{otherwise.}
\end{cases}
\]
The diffusion $\varepsilon$ and regularization parameter $\alpha$ will vary.
\end{modelproblem}

\begin{modelproblem}[Unit square and constant convection with limited observation]\label{mp:1a}
    Let $\Omega = (0,1)^2$ be the computational domain and $\mathcal O = (\tfrac14,\tfrac34)^2$
    be the observation domain. The remainder of this problem is the
    same as for Model problem~\ref{mp:1}.
\end{modelproblem}

Furthermore, we consider a non-trivial geometry $\Omega$, which is a approximation of a quarter annulus by means of a B-spline parameterization, see Figure~\ref{fig:2DDomain} (right).
Again, Model problem~\ref{mp:2} is a problem with full observation and the observation domain in Model problem~\ref{mp:2a} is the smaller area in Figure~\ref{fig:2DDomain} (right).
We observe that for this domain, the conditions of Theorem~\ref{theorm:conformingSplines} are not satisfied.

\begin{modelproblem}[Quarter annulus and varying convection with full observation]\label{mp:2}
    Let $\Omega = \mathcal O = \mathbf{G}((0,1)^2)$ with $\mathbf{G}:(0,1)^2 \rightarrow \mathbb{R}^2 $ and
    \begin{equation}\label{eq:G}
        \mathbf{G}(\widehat{x}) = 
        \begin{pmatrix}
        (1 + \widehat{x}_1) (1-\widehat{x}_2) (1 + 2 (\sqrt2-1) \widehat{x}_2)\\
        (1 + \widehat{x}_1) \widehat{x}_2 (2 \sqrt2 -1- 2 (\sqrt2-1) \widehat{x}_2)

        \end{pmatrix}
    \end{equation}
    be the computational domain, which is also
    the observation domain. The convection is $\beta = (y,1+x^2)$ and there is no reaction $\sigma = 0$ or source term $f = 0$. The desired state is  
\[
    u_d(x,y)= 
\begin{cases}
    1, \quad &\text{if}\quad (x-x_0)^2 + (y-y_0)^2 \leq \frac{1}{16}\\
    0, \quad    & \text{otherwise,}
\end{cases}
\]
where $(x_0,y_0) = \mathbf{G}(\frac38,\frac58)$.
The diffusion $\varepsilon$ and regularization parameter $\alpha$ will vary.
\end{modelproblem}

\begin{modelproblem}[Quarter annulus and varying convection with limited observation]\label{mp:2a}
    Let $\Omega = \mathbf{G}((0,1)^2)$ be the computational domain and $\mathcal O = \mathbf{G}((\tfrac14,\tfrac34)^2)$ be the observation domain, where $\textbf G$ is as in~(\ref{eq:G}). The remainder of this problem is the same as for Model problem~\ref{mp:2}.
\end{modelproblem}

For all model problems, we consider a discretization of the optimality system using the spaces given in~(\ref{eq:stateSpaceDisc}) as outlined in Section~\ref{sec:IGA}. The resulting linear system of equations
\[
\mathcal{A}_h \underline{\mathbf{x}}_h = \underline{\mathbf{b}}_h 
\]
is solved using the MINRES method, preconditioned with the proposed Schur complement preconditioner (\ref{eq:SchurNormDisc}). We use a random initial guess $\underline{\textbf{x}}_{h,0}$. The stopping criterion is 
\[
\|\underline{\mathbf{r}}_k\| \leq 10^{-8}\|\mathbf{\underline{r}}_0\|,
\]
where $\underline{\mathbf{r}}_k := \underline{\mathbf{b}}_h - \mathcal{A}_h  \underline{\mathbf{x}}_{h,k}$ denotes the residual and $\| \cdot \|$ is the Euclidean norm. 

\subsection{Results with exact preconditioner}
\label{sec:num1}
In this section, we present the results for the Schur complement preconditioner (\ref{eq:SchurNormDisc})
when realized using a sparse Cholesky decomposition. 

Table~\ref{t:ItNumParaDiff} shows the iteration numbers needed to reach the stopping criteria for full and partial observation (Model problems \ref{mp:1} and \ref{mp:1a}). In these tables, $\alpha$ and $\varepsilon$ are varied, while $p = 2$ and $\ell = 6$ are fixed. In Table~\ref{t:ItNumParah}, we set $\varepsilon = 10^{-3}$ and vary the refinement level $\ell$ and $\alpha$. From the tables, we observe that for the partial observation problem, we need a few more iterations for small values of $\alpha$. This is probably because $M_{\mathcal{O},h}$ is singular in case of
partial observation.
The iteration numbers are relatively small for all considered values of $\alpha$, $\varepsilon$ and $\ell$. This is predicted by the theory as Model problems~\ref{mp:1} and \ref{mp:1a} satisfy the conditions of Theorem~\ref{theorm:conformingSplines}.
\begin{table}[ht]
{\footnotesize
\caption{Iteration numbers: Model problem \ref{mp:1} (left) and \ref{mp:1a} (right),  $p = 2$, $\ell = 6$.}\label{t:ItNumParaDiff}
\begin{center}
  \begin{tabular}{| c || c | c | c | c | }
    \hline
    $\varepsilon \;\backslash\; \alpha$ & $10^0$ & $10^{-3}$ & $10^{-6}$ & $10^{-9}$ \\ \hline \hline
    $10^{0}$  &     12 & 26& 60& 72  \\  \hline   
    $10^{-3}$  &   15 & 47& 26& 11  \\  \hline
    $10^{-6}$  &   15 & 46& 26& 11  \\  \hline
    $10^{-9}$  &   15 & 46& 26& 11  \\  \hline
  \end{tabular}$\quad$
  \begin{tabular}{| c || c | c | c | c | }
    \hline
    $\varepsilon \;\backslash\; \alpha$ & $10^0$ & $10^{-3}$ & $10^{-6}$ & $10^{-9}$ \\ \hline \hline
    $10^{0}$  &     12 & 20& 57& 78  \\  \hline   
    $10^{-3}$  &   15 & 41& 54& 19  \\  \hline
    $10^{-6}$  &   14 & 41& 53& 19  \\  \hline
    $10^{-9}$  &   14 & 41& 53& 19  \\  \hline
  \end{tabular}
\end{center}
}
\end{table}
\begin{table}[ht]
{\footnotesize
\caption{Iteration numbers: Model problem \ref{mp:1} (left) and \ref{mp:1a} (right),  $p = 2$, $\varepsilon = 10^{-3}$.}\label{t:ItNumParah}
\begin{center}
  \begin{tabular}{| c || c | c | c | c | }
    \hline
    $\ell \;\backslash\; \alpha$ & $10^0$ & $10^{-3}$ & $10^{-6}$ & $10^{-9}$ \\ \hline \hline

    $4$  &   15 & 46& 14& 7  \\  \hline
    $5$  &   15 & 47& 19& 8  \\  \hline
    $6$  &   15 & 47& 26& 11  \\  \hline
    $7$  &   15 & 46& 38& 11  \\  \hline
  \end{tabular}$\quad$
  \begin{tabular}{| c || c | c | c | c | }
    \hline
    $\ell \;\backslash\; \alpha$ & $10^0$ & $10^{-3}$ & $10^{-6}$ & $10^{-9}$ \\ \hline \hline

    $4$  &   15 & 46& 35& 15  \\  \hline
    $5$  &   15 & 43& 44& 17  \\  \hline
    $6$  &   15 & 41& 54& 19  \\  \hline
    $7$  &   15 & 39& 55& 22  \\  \hline
  \end{tabular}
\end{center}
}
\end{table}

Next, we consider Model problems~\ref{mp:2} and~\ref{mp:2a}, which are the problems where the computational domain is a quarter annulus. The iteration numbers are shown in Tables~\ref{t:ItNumQADiff} and \ref{t:ItNumQAh}. Note that the conditions of Theorem~\ref{theorm:conformingSplines} are not satisfied. Nevertheless, the iteration numbers are comparable with those of Tables~\ref{t:ItNumParaDiff} and \ref{t:ItNumParah}.
\begin{table}[ht]
{\footnotesize
\caption{Iteration numbers: Model problem \ref{mp:2} (left) and \ref{mp:2a} (right), $p = 2$, $\ell = 6$.}\label{t:ItNumQADiff}
\begin{center}
  \begin{tabular}{| c || c | c | c | c | }
    \hline
    $\varepsilon \;\backslash\; \alpha$ & $10^0$ & $10^{-3}$ & $10^{-6}$ & $10^{-9}$ \\ \hline \hline
    $10^{0}$   &   17 & 41& 62& 64  \\  \hline   
    $10^{-3}$  &   18 & 48& 29& 11  \\  \hline
    $10^{-6}$  &   18 & 48& 29& 11  \\  \hline
    $10^{-9}$  &   18 & 48& 29& 11  \\  \hline
  \end{tabular}$\quad$ 
  \begin{tabular}{| c || c | c | c | c | }
    \hline
    $\varepsilon \;\backslash\; \alpha$ & $10^0$ & $10^{-3}$ & $10^{-6}$ & $10^{-9}$ \\ \hline \hline
    $10^{0}$   &   17 & 32& 60& 76  \\  \hline   
    $10^{-3}$  &   17 & 46& 54& 25  \\  \hline
    $10^{-6}$  &   17 & 46& 54& 25  \\  \hline
    $10^{-9}$  &   17 & 46& 54& 25  \\  \hline
  \end{tabular}
\end{center}
}
\end{table}
\begin{table}[ht]
{\footnotesize
\caption{Iteration numbers: Model problem \ref{mp:2} (left) and \ref{mp:2a} (right),  $p = 2$, $\varepsilon = 0.001$.}\label{t:ItNumQAh}
\begin{center}
  \begin{tabular}{| c || c | c | c | c | }
    \hline
    $\ell \;\backslash\; \alpha$ & $10^0$ & $10^{-3}$ & $10^{-6}$ & $10^{-9}$ \\ \hline \hline

    $4$  &   18 & 47& 16& 10 \\  \hline
    $5$  &   18 & 48& 21& 11 \\  \hline
    $6$  &   18 & 48& 29& 11  \\  \hline
    $7$  &   16 & 48& 42& 12  \\  \hline
  \end{tabular}$\quad$
  \begin{tabular}{| c || c | c | c | c | }
    \hline
    $\ell \;\backslash\; \alpha$ & $10^0$ & $10^{-3}$ & $10^{-6}$ & $10^{-9}$ \\ \hline \hline

    $4$  &   16 & 48& 40& 25  \\  \hline
    $5$  &   16 & 46& 46& 25  \\  \hline
    $6$  &   17 & 46& 54& 25  \\  \hline
    $7$  &   16 & 46& 55& 28  \\  \hline
  \end{tabular}
\end{center}
}
\end{table}

\begin{remark}
For $\varepsilon = 1$, the iteration numbers in Table~\ref{t:ItNumParaDiff} (and Table~\ref{t:ItNumQADiff}) are growing as $\alpha$ becomes smaller. This may appear strange since we have proven that the condition number (for Table~\ref{t:ItNumParaDiff}) is less then 4.05. Describing convergence estimates for Krylov subspace methods in term of only the condition number can be misleading and/or insufficient, cf. \cite{malek2014preconditioning}.
The different iteration numbers for various values of $\varepsilon$ and $\alpha$ can be explained by the distribution of the eigenvalues. For small iteration numbers the eigenvalues are more clustered. For $\varepsilon = 1$, the iteration numbers starts decreasing when $\alpha< 10^{-9}$.
\end{remark}

\subsection{Results with inexact preconditioner}
\label{sec:num2}

So far, we have realized the proposed preconditioners using sparse direct solvers.
This approach works well for mid-sized problems. For large-sized problems, alternatives are of interest since they might be faster
or have a smaller memory footprint. We replace $\mathcal{S}_h$ by a spectrally equivalent approximation $\widetilde{\mathcal{S}}_h$, where the action of $\widetilde{\mathcal{S}}_h^{-1}$ can be calculated efficiently.
The spectral equivalence should be robust with respect to the parameters of interest.

For the approximation of the mass matrix $M_h$, which is found in the first and the second block
of the overall preconditioner, we exploit the fact that the mass matrix on the parameter domain is
the Kronecker product of two mass matrices that correspond to the discretization of a univariate 
problem, i.e., we have
\[
    M_h = M^{(1)} \otimes M^{(2)},
\]
where $\otimes$ denotes the Kronecker product and $M^{(1)}$ and $M^{(2)}$ denote the univariate mass matrices.
For the Model problems~\ref{mp:2} and~\ref{mp:2a}, we use a similar preconditioner that incorporates a tensor-rank-1 approximation of the geometry, which is derived as follows. As common in IgA, the bilinear forms are computed by transformation to the parameter domain, i.e., we have
\[
        (u,v)_{L_2(\Omega)} = \int_0^1 \int_0^1 J(x_1,x_2) \,u(x_1,x_2) \,v(x_1,x_2) \,\mathrm d x_2 \,\mathrm d x_1,
\]
where $J(x) = |\mbox{det} \nabla \textbf G(x)|$, which we
approximate by 
\[
         (u,v)_{\widetilde M} := \frac{1}{J(\tfrac12,\tfrac12)}
        \int_0^1 \int_0^1 J(x_1,\tfrac12)\,J(\tfrac12,x_2)\,u(x_1,x_2) \,v(x_1,x_2) \,\mathrm d x_2 \,\mathrm d x_1.
\]
The corresponding mass matrix $\widetilde M_h$ has tensor-product structure:
\[
    \widetilde M_h = \widetilde M^{(1)} \otimes \widetilde M^{(2)},
\]
where $\widetilde M^{(1)}$ and $\widetilde M^{(2)}$ denote the univariate mass matrices, which are lumped with $J(x_1,\tfrac12)$ and $J(\tfrac12,x_2)$, respectively. Straight-forward computations
show that the relative condition number of the exact mass matrix and its approximation can be bounded uniformly by a term that only depends on $\textbf{G}$. For realizing the inverse of $\widetilde M_h$ efficiently, we make use of the fact that the application of the inverse of a Kronecker product to some vector can be efficiently realized using sparse direct solvers that realize the application of $(\widetilde M^{(1)})^{-1}$ and $(\widetilde M^{(2)})^{-1}$.

For the approximation of the inverse of the matrix $M_{\mathcal O,h}+ \alpha B_h$, representing a fourth-order
PDE, we use a geometric multigrid solver. Following the standard approach, we assume to have a hierarchy of quasi-uniform grids, where the grid sizes of two consecutive grids differ by a factor of two. Since we have tensor-product grids in Isogeometric Analysis, such a grid hierarchy can be easily constructed by coarsening. The coarsest grid level is chosen such that there are no inner knots.
On each of these grid levels $\ell=0,1,\ldots,L$, we introduce a discretization space $U_{h_\ell} = S_{p,\ell}(\Omega)\cap H^1_0(\Omega)$. One iterate of the multigrid solver consists of the following steps:
\begin{itemize}
    \item Apply $\nu=2$ forward Gauss-Seidel sweeps as pre-smoother.
    \item Apply coarse-grid correction. Since we have nested grids ($U_{h_\ell}\subset U_{h_{\ell+1}}$), the coarse-grid correction is realized based on canonical embedding. For $\ell>1$, the problem on the next coarser level $\ell-1$ is solved by applying $1$ step of the multigrid method recursively (V-cycle). Only on the coarsest grid level $\ell=0$, the problem is solved using a direct solver.
    \item Apply $\nu=2$ backward Gauss-Seidel sweeps as post-smoother.
\end{itemize}

The robustness of that multigrid method in the grid size is a straight-forward extension of
the known results for the biharmonic problem, cf.~\cite{sogn2018schur}.
It is worth mentioning that this argument does not cover
the robustness in any of the other parameters that affect the multigrid solver.

We again use a MINRES solver, now preconditioned with the presented
tensor-rank-one approximation of the mass matrices and with one step of the multigrid solver. The corresponding numerical results are presented in Tables~\ref{t:ItNumQAGSDiff} and \ref{t:ItNumQAGSh}. In Table~\ref{t:ItNumQAGSDiff}, we observe that the iteration counts are uniformly bounded for all choices of $\epsilon$ and $\alpha$, however with much larger values than for the exact preconditioner. This is related to the well-known fact that standard Gauss-Seidel smoothers do not perform well in the framework of Isogeometric Analysis. The convergence deteriorates particularly if the spline degree is increased, which can also be seen in  Table~\ref{t:ItNumQAGSh}. Furthermore, in Table~\ref{t:ItNumQAGSh}, we can also study the dependence of the convergence on the grid size. Although, the convergence theory predicts a robust convergence behavior, this is not observed in practice for the grid levels considered. Apparently, this is the case since the Gauss-Seidel smoother does not work well for spline bases, even for moderate values of $p$, cf., e.g.,~\cite{sogn2019robust}.

\begin{table}[ht]
{\footnotesize
\caption{Iteration numbers: Model problem \ref{mp:2} (left) and \ref{mp:2a} (right), $p = 2$, $\ell = 6$.}\label{t:ItNumQAGSDiff}
\begin{center}
  \begin{tabular}{| c || c | c | c | c | }
    \hline
    $\varepsilon \;\backslash\; \alpha$ & $10^0$ & $10^{-3}$ & $10^{-6}$ & $10^{-9}$ \\ \hline \hline
    $10^{0}$   & 128 & 108& 104& 67  \\  \hline   
    $10^{-3}$  & 169 & 81& 42& 26  \\  \hline
    $10^{-6}$  & 179 & 81& 42& 26  \\  \hline
    $10^{-9}$  & 179 & 81& 42& 26  \\  \hline
  \end{tabular}$\quad$
  \begin{tabular}{| c || c | c | c | c | }
    \hline
    $\varepsilon \;\backslash\; \alpha$ & $10^0$ & $10^{-3}$ & $10^{-6}$ & $10^{-9}$ \\ \hline \hline
    $10^{0}$   &   128& 104& 132& 175  \\  \hline   
    $10^{-3}$  &   171& 112& 139& 172  \\  \hline
    $10^{-6}$  &   178& 112& 142& 174  \\  \hline
    $10^{-9}$  &   178& 112& 142& 174  \\  \hline
  \end{tabular}
\end{center}
}
\end{table}

\begin{table}[ht]
{\footnotesize
\caption{Iteration numbers: Model problem \ref{mp:2} (left) and \ref{mp:2a} (right), $\alpha = 0.001$, $\varepsilon = 0.001$.}\label{t:ItNumQAGSh}
\begin{center}
  \begin{tabular}{| c || c | c | c | c | }
    \hline
    $\ell \;\backslash\; p$ & $2$ & $3$ & $5$ & $7$ \\ \hline \hline
    $4$  & 49 & 64& 191& 730  \\  \hline   
    $5$  & 55 & 61& 150& 510  \\  \hline
    $6$  & 81 & 86& 137& 440  \\  \hline
    $7$  & 118 &134&179& 380  \\  \hline
  \end{tabular}$\quad$
  \begin{tabular}{| c || c | c | c | c | }
    \hline
    $\ell\;\backslash\; p$ & $2$ & $3$ & $5$ & $7$ \\ \hline \hline
    $4$  & 57& 63& 152& 567 \\  \hline   
    $5$  & 77& 85& 135& 437 \\  \hline
    $6$  & 112& 123& 170& 414 \\  \hline
    $7$  & 152& 162& 204& 364 \\  \hline
  \end{tabular}
\end{center}
}
\end{table}

To obtain a better convergence behavior, we consider a second approach for the smoother: a macro Gauss-Seidel approach. This approach makes
use of the tensor-product structure of the discretization. For two dimensions, the degrees of freedom can be represented as a grid in the plane, see Figure~\ref{fig:macroGS} (left). Each dot represents one degree of freedom or basis function. We start by introducing a macro grid that groups $a\times a$ degrees of freedom. (If the number of rows or columns is not divisible by $a$, the last macro elements in each direction are correspondingly smaller.) The macro grid is depicted in Figure~\ref{fig:macroGS} (left).

Each of the macro elements consists of the degrees of freedom that belong to the element of the macro grid and of degrees of freedom of the neighboring elements of the macro grid. Here, we use $b$ additional rows and columns each on each of the sides, see Figure~\ref{fig:macroGS} (right).

\begin{figure}[h]
  \center
  \includegraphics[width=0.35\textwidth]{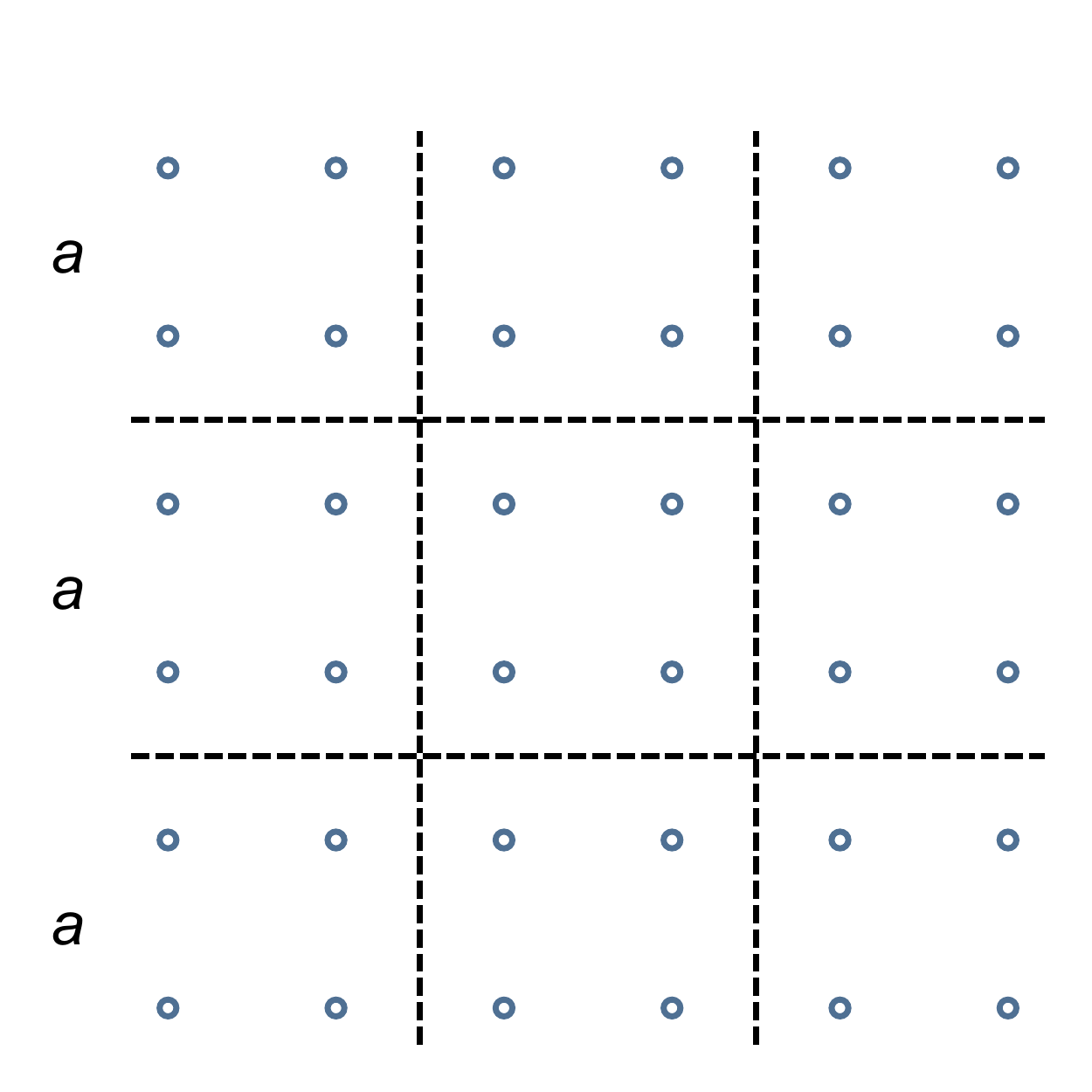}\quad
  \includegraphics[width=0.35\textwidth]{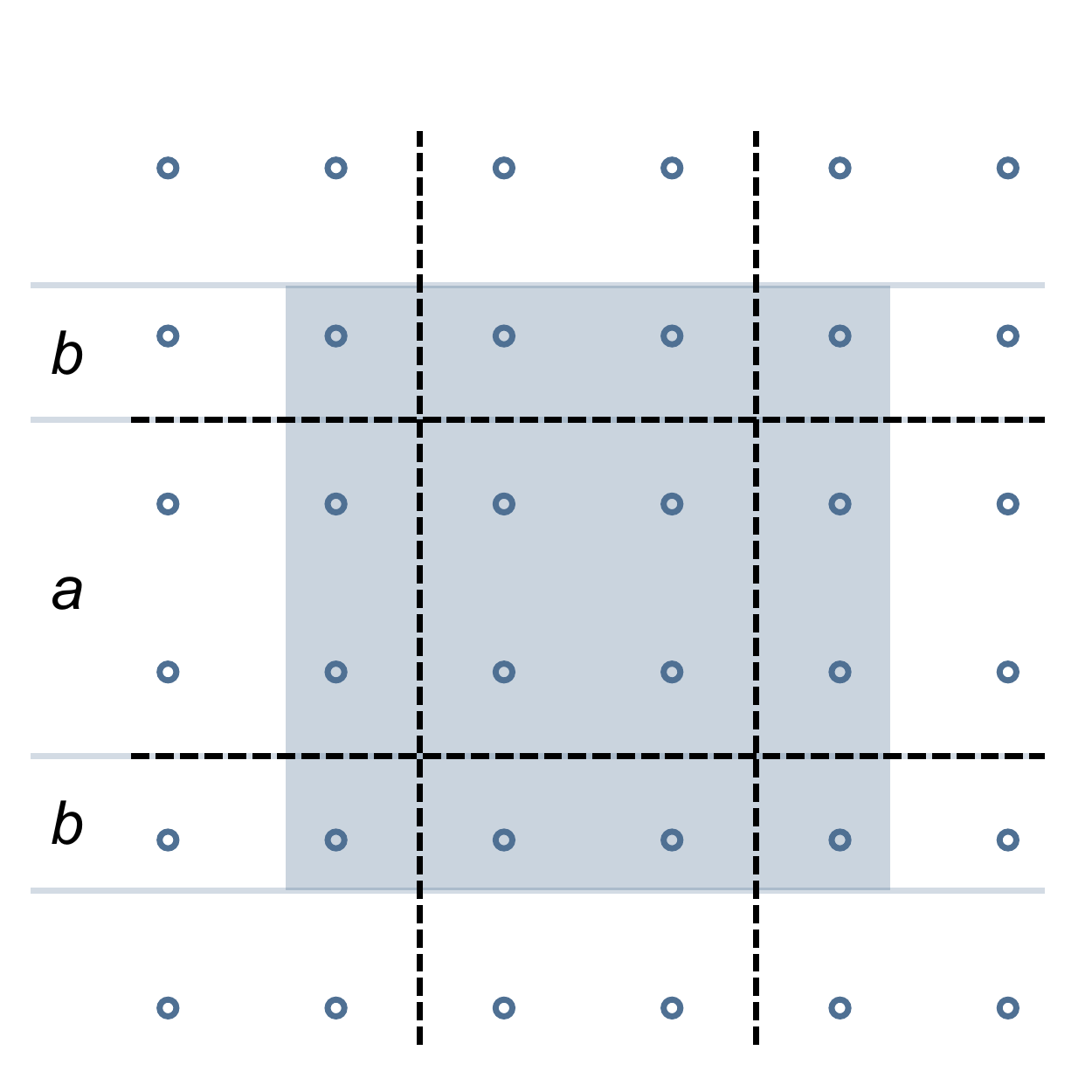}
  \caption{The construction of the macro Gauss-Seidel approach.}
  \label{fig:macroGS}
\end{figure}

Then, a macro Gauss-Seidel sweep is a standard multiplicative Schwarz method, where the subspaces are the degrees of freedom that belong to each of the macro elements. So, the choice $a:=1$ and $b:=0$ corresponds to a standard Gauss-Seidel sweep.

In the following, we use the patch size $a:=p$ and the overlap size $b:=p-1$. As for the standard Gauss-Seidel case, we apply a forward sweep for pre-smoothing and a backward sweep, i.e., with the reverse ordering of the macro elements, for post-smoothing. The problem within the (relatively small) subspaces is solved by means of a direct solver. The number of smoothing steps $\nu$ is set to $1$.

The corresponding iteration counts are presented in Tables~\ref{t:ItNumQAMGSDiff} and \ref{t:ItNumQAMGSh}. In all cases, we obtain significantly better convergence rates than for a standard Gauss-Seidel smoother. Table~\ref{t:ItNumQAMGSh} shows that the resulting method is robust in the spline degree, and that the method is quite robust in the grid size. Table~\ref{t:ItNumQAMGSDiff} shows that the overall method is also robust in the parameter $\varepsilon$ and well-bounded for $\alpha$.

\begin{table}[ht]
{\footnotesize
\caption{Iteration numbers: Model problem \ref{mp:2} (left) and \ref{mp:2a} (right), $p = 2$, $\ell = 6$.}\label{t:ItNumQAMGSDiff}
\begin{center}
  \begin{tabular}{| c || c | c | c | c | }
    \hline
    $\varepsilon \;\backslash\; \alpha$ & $10^0$ & $10^{-3}$ & $10^{-6}$ & $10^{-9}$ \\ \hline \hline
    $10^{0}$   & 50 & 51& 62& 64  \\  \hline   
    $10^{-3}$  & 91 & 49& 29& 13  \\  \hline
    $10^{-6}$  & 96 & 49& 29& 13  \\  \hline
    $10^{-9}$  & 96 & 49& 29& 13  \\  \hline
  \end{tabular}$\quad$
  \begin{tabular}{| c || c | c | c | c | }
    \hline
    $\varepsilon \;\backslash\; \alpha$ & $10^0$ & $10^{-3}$ & $10^{-6}$ & $10^{-9}$ \\ \hline \hline
    $10^{0}$  & 50 & 46 & 72& 98  \\  \hline   
    $10^{-3}$ & 94 & 77 & 99& 103  \\  \hline
    $10^{-6}$ & 96 & 77 & 99& 103  \\  \hline
    $10^{-9}$ & 96 & 77 & 99& 103  \\  \hline
  \end{tabular}
\end{center}
}
\end{table}

\begin{table}[ht]
{\footnotesize
\caption{Iteration numbers: Model problem \ref{mp:2} (left) and \ref{mp:2a} (right), $\alpha = 0.001$, $\varepsilon = 0.001$.}\label{t:ItNumQAMGSh}
\begin{center}
  \begin{tabular}{| c || c | c | c | c | }
    \hline
    $\ell \;\backslash\; p$ & $2$ & $3$ & $5$ & $7$ \\ \hline \hline
    $4$  & 47& 48& 48& 48  \\  \hline   
    $5$  & 48& 48& 48& 48  \\  \hline
    $6$  & 49& 48& 48& 48  \\  \hline
    $7$  & 64& 49& 48& 48  \\  \hline
  \end{tabular}$\quad$
  \begin{tabular}{| c || c | c | c | c | }
    \hline
    $\ell\;\backslash\; p$ & $2$ & $3$ & $5$ & $7$ \\ \hline \hline
    $4$  & 49& 48& 46& 48  \\  \hline   
    $5$  & 57& 52& 46& 46  \\  \hline
    $6$  & 77& 64& 53& 49  \\  \hline
    $7$  & 96& 80& 58& 52  \\  \hline
  \end{tabular}
\end{center}
}
\end{table}

\appendix

\section{Proof of Lemma~\ref{lemma:SFI}}
\label{sec:proof}
The inequality in Lemma~\ref{lemma:SFI} is sometimes referred to as \textit{the second fundamental inequality}, cf. \cite{ladyzhenskaya2013boundary}. For domains with polygonal (polyhedral) Lipschitz boundary the result is known, but for domains which are images of geometry mappings we were unable to find any result. We therefore provide a proof in this appendix. 
We start with providing a density result.
\begin{lemma}
\label{lemma:H3isDinH2}
Let the domain $\Omega$ have a Lipschitz boundary and be the image of a geometric mapping $\mathbf{G}:\widehat{\Omega}:=(0,1)^d\rightarrow \Omega$, where both
$\|\nabla^r \mathbf{G}\|_{L^\infty}$ and
$\|(\nabla^r \mathbf{G})^{-1}\|_{L^\infty}$ are bounded for $r\in\{1,2,3\}$, then $H^3(\Omega)\cap H^1_0(\Omega)$ is dense in $H^2(\Omega)\cap H^1_0(\Omega)$.
\end{lemma}
\begin{proof}
Let $V := H^3(\Omega)\cap H^1_0(\Omega)$ and $U := H^2(\Omega)\cap H^1_0(\Omega)$, we want that for any $\epsilon > 0$ and $u\in U$, there exist a $v\in V$ such that
\begin{equation}
\label{eq:DensePhys}
 \|u-v\|_{H^2(\Omega)} \le \epsilon.
\end{equation}
From \cite[Theorem 1.6.2]{grisvard1992singularities} we know that $\widehat V := H^3(\widehat \Omega)\cap H^1_0(\widehat \Omega)$ is dense in $\widehat U := H^2(\widehat \Omega)\cap H^1_0(\widehat \Omega)$, i.e., for any $\epsilon > 0$ and $u\in \widehat U$, there exist a $\widehat v \in \widehat V$ such that
\begin{equation}
\label{eq:DensePara}
  \|\widehat u-\widehat v\|_{H^2(\widehat \Omega)} \le \epsilon.
\end{equation}
Now, we know using the standard IgA-results that
\begin{equation}
\label{eq:IGAH2}
\|w\|_{H^2(\Omega)}	\le c_g \|w \circ \mathbf{G}\|_{H^2(\widehat \Omega)} 
\end{equation}
holds for all $w\in H^2(\Omega)$, where $c_g$ only depends on the geometry.
Moreover, we have
\[
v\in V \Leftrightarrow v\circ \mathbf{G} \in \widehat V \qquad\mbox{and}\qquad u\in U \Leftrightarrow u\circ \mathbf{G} \in \widehat U.
\]
Now we prove \eqref{eq:DensePhys}. Let $u\in U$ and $\epsilon > 0$ be given. Let $\widehat u:= u\circ \mathbf{G}$. Using \eqref{eq:DensePara},	we know that there exist a $\widehat v\in V_g$ such that
\[
\|\widehat u-\widehat v\|_{H^2(\widehat \Omega)} \le \epsilon/c_g.
\]
By choosing $v:=\widehat v\circ \mathbf{G}^{-1}$, we have
\[
\|( u- v)\circ \mathbf{G}\|_{H^2(\widehat \Omega)} \le \epsilon/c_g.
\]
and using \eqref{eq:IGAH2} consequently
\[
\| u- v\|_{H^2(\Omega)} \le \epsilon.
\]
This means that we have found a proper $v\in V$ such that \eqref{eq:DensePhys} holds.
\end{proof}
Next, we state a weighted trace theorem \cite[Theorem 1.5.1.10]{grisvard2011elliptic}.
\begin{theorem}
\label{theo:traceScaled}
Let $\Omega$ be a bounded open subset of $\mathbb{R}^d$ with Lipschitz boundary and let $T$ be the trace operator. Then there a exist a constant $c$ which only depend  on $\Omega$ such that \begin{equation}
    \int_{\partial \Omega} |Tv(x)|^2 \ d s \leq c\left[ \sqrt{\delta}\int_{ \Omega} |\nabla v(x)|^2 \ d x + \frac{1}{\sqrt{\delta}}\int_{ \Omega} |v(x)|^2 \ d x \right],
\end{equation}
hold for all $v(x)\in H^1(\Omega)$ and all $\delta \in (0,1)$.
\end{theorem}
We can now prove Lemma~\ref{lemma:SFI}.
\begin{proof}
When $\Omega$ has a polygonal (polyhedral) Lipschitz boundary the result follows from  \cite{grisvard1992singularities,grisvard2011elliptic}. A detailed proof of this case can be found in \cite[Lemma 3.3]{SogZul18}. We consider the case where $\Omega$ is the image of a geometric mapping and has Lipschitz boundary. According to \cite[Theorem 3.1.1.2]{grisvard2011elliptic} we have
\[
  \int_\Omega |\nabla \cdot \psi(x)|^2 \ d x = \int_\Omega \nabla \psi(x) : (\nabla \psi(x))^T \ dx - \int_{\partial\Omega} g(x) (\psi_n(x))^2\ ds,
\]
for all $\psi \in H^2(\Omega)^d$ with $\psi_n := \psi \cdot n$ and $\psi_T := \psi - \psi_n \, n = 0$. Here, $g(x)$ is a function which depends on the curvature of boundary $\partial \Omega$. This can be bounded from above by a constant $c_g$ depending only on $\Omega$:
\begin{equation}
\label{eq:Grisvard3112}
  \int_\Omega |\nabla \cdot \psi(x)|^2 \ d x \geq \int_\Omega \nabla \psi(x) : (\nabla \psi(x))^T \ dx - c_g\int_{\partial\Omega} (\psi_n(x))^2\ ds.
\end{equation}
Applying this inequality to $\psi = \nabla u$ with $u \in H^3(\Omega)\cap H_0^1(\Omega)$, we now bound the last term by using Theorem~\ref{theo:traceScaled}
\begin{align*}
    - c_g\int_{\partial\Omega} (\psi_n(x))^2\ ds &=  - c_g\int_{\partial\Omega} (\nabla u(x)\cdot n)^2\ ds \geq  - c_g d\int_{\partial\Omega} |\nabla u(x)|^2\ d s\\
    &\geq -c \left[\sqrt{\delta}\int_{ \Omega} |\nabla^2 u(x)|^2 \ d x + \frac{1}{\sqrt{\delta}}\int_{ \Omega} |\nabla u(x)|^2 \ d x\right].
\end{align*}
By using integration by parts, the Cauchy--Schwarz inequality and the Poincar\' e inequality, we can bound the last term by
\begin{align*}
\int_{ \Omega}  |\nabla u(x)|^2 \ d x \leq c^2_P\|\Delta u\|^2_{L^2(\Omega)},
\end{align*}
where $c_P$ is the Poincar\' e constant. Combining the last two inequalities gives
\[
 - c_g\int_{\partial\Omega} (\psi_n(x))^2\ ds \geq -c \left[\sqrt{\delta} \|\nabla^2 u\|^2_{L^2(\Omega)}  + \frac{c^2_P}{\sqrt{\delta}} \|\Delta u\|^2_{L^2(\Omega)}\right].
\]
Inserting the inequality above and $\psi=\nabla u$ into \eqref{eq:Grisvard3112} gives
\[
\|\Delta u \|^2_{L^2(\Omega)} \geq \|\nabla^2 u \|^2_{L^2(\Omega)} - c \left[\sqrt{\delta} \|\nabla^2 u\|^2_{L^2(\Omega)} + \frac{c^2_P}{\sqrt{\delta}} \|\Delta u\|^2_{L^2(\Omega)}\right].
\]
Note that this holds for any $\delta \in (0,1)$. We now choose $\delta$ such that $1-c\sqrt{\delta}$ is positive and we get
\[
\|\nabla^2 u \|_{L^2(\Omega)}  \leq c \,\|\Delta u \|^2_{L^2(\Omega)} \quad \foralls u\in H^3(\Omega) \cap H^1_0(\Omega). 
\]
We that note due to the boundary condition and the Poincar\' e inequality it follows that $\|\nabla^2 u \|_{L^2(\Omega)}$ is equivalent to the $H^2$-norm. So, we have now shown inequality~\eqref{eq:2FAPE} for $u\in H^3(\Omega) \cap H^1_0(\Omega)$. Since $H^3(\Omega) \cap H^1_0(\Omega)$ is dense in $H^2(\Omega) \cap H^1_0(\Omega)$ (Lemma~\ref{lemma:H3isDinH2}) the result also holds for all $u\in H^2(\Omega) \cap H^1_0(\Omega)$.
\end{proof}

\section{Approximation error estimates for B-splines}
\label{sec:app2}
In this Appendix, we prove Theorem~\ref{theo:errorEst} and some auxiliary results required for that proof. We consider B-splines with maximum smoothness on the parameter domain $\widehat{\Omega}:=(0,1)^d$, that is, we consider the space $S^d_{p,\ell}$. We point out that for functions in $H^2((0,1)^d) \cap H^1((0,1)^d)$ the $H^2$-semi-norm and $L^2$-norm of the Laplacian coincide, that is,
\[
\|\Delta u \|_{L^2{(0,1)^d}} = \|\nabla^2 u \|_{L^2{(0,1)^d}}\quad \foralls H^2((0,1)^d) \cap H^1((0,1)^d).
\]
For any $d\in \mathbb{N}$ and $p\in \mathbb N$ with $p \geq 3$, let $\mathbf{\Pi}_p: H^2((0,1)^d)\cap H^1_0((0,1)^d) \rightarrow S^d_{p,\ell}\cap H^1_0((0,1)^d)$ be the $H^2$-orthogonal projector, defined by
\begin{align*}
    \inner{\Delta \mathbf{\Pi}_p u}{\Delta\tilde{u}}_{L^2((0,1)^d)} &=  \inner{\Delta u}{\Delta\tilde{u}}_{L^2((0,1)^d)} \quad \foralls \tilde{u} \in S^d_{p,\ell}\cap  H^1_0.
\end{align*}
To better distinguish the univariate case ($d=1$), we write $\Pi_p := \mathbf{\Pi}_p$ for that case.

\begin{theorem}
\label{theo:L2H2dD}
Let $d\in \mathbb{N}$ and $p\in \mathbb{N}$ with $p\geq 3$. Then there exits a constant $c>0$ such that
\[
\|u - \mathbf{\Pi}_p u\|_{L^2((0,1)^d)} \leq c\,h^2 \|\Delta u\|_{L^2((0,1)^d)}\quad \foralls u\in H^2((0,1)^d)\cap H^1_0((0,1)^d).
\]
\end{theorem}
\begin{proof}
Let $u\in H^2((0,1)^d)\cap H^1_0((0,1)^d)$ be arbitrary but fixed.
\cite[Theorem 9.3]{sogn2018schur} states that
\[
    \|u - \widetilde{\mathbf{\Pi}}_p u\|_{L^2((0,1)^d)} \leq c\,h^2 \|\Delta u\|_{L^2((0,1)^d)},
\]
where $\widetilde{\mathbf{\Pi}}_p: H^2((0,1)^d)\cap H^1_0((0,1)^d) \rightarrow \widetilde{S}$ is the $H^2$-orthogonal projector into some
space $\widetilde{S}\subset S^d_{p,\ell}$. Using $\mathbf{\widetilde{\Pi}}_p \mathbf{\Pi}_p = \mathbf{\widetilde{\Pi}}_p$, the triangle inequality and the stability statement
$\|\Delta \mathbf{\Pi}_p u\|_{L^2((0,1)^d)}\le \|\Delta u\|_{L^2((0,1)^d)}$, we immediately obtain the desired result.
\end{proof}
Next, we provide an $H^2$--$H^4$ error estimate for the univariate case.
\begin{theorem}
\label{theo:H2H41D}
Let $p\in \mathbb{N}$ with $p\geq 3$. Then, 
\[
\|\partial^2 (u - \Pi_p u)\|_{L^2(0,1)} \leq 2h^2 \|\partial^4 u\|_{L^2(0,1)}\quad \foralls H^4(0,1)\cap H^1_0(0,1).
\]
\end{theorem}
\begin{proof}
See \cite[Theorem~3]{sogn2019robust}. 
\end{proof}
We define projectors $\Pi_p^{x_k}$ on $C^{\infty}((0,1)^d)$ as follows:
\begin{align*}
&(\Pi_p^{x_k})u(x_1,\ldots,x_{k-1},\cdot,x_{k+1},\ldots,x_{d}) := \Pi_p u(x_1,\ldots,x_{k-1},\cdot,x_{k+1},\ldots,x_{d})\\
&\hspace{2.6cm}\forall\, (x_1,\ldots,x_{k-1},x_{k+1},\ldots,x_{d})\in(0,1)^{d-1}\quad \text{for} \quad k = 1,\ldots,d .
\end{align*}
These projectors act on one variable. We also introduce
projectors $\mathbf{\Pi}_p^{x_k}$ that act on every variable except one,
which are given by
\begin{align*}
&(\mathbf{\Pi}_p^{x_k})u(\cdot,\ldots,\cdot,x_k,\cdot,\ldots,\cdot) := \mathbf{\Pi}_p u(\cdot,\ldots,\cdot,x_k,\cdot,\ldots,\cdot)\\
&\hspace{2.6cm}\forall\, x_k\in(0,1)\quad \text{for} \quad k = 1,\ldots,d .
\end{align*}
Similarly we
define a Laplace operator on the form
\[
\Delta^{x_k} = 
    \sum_{i \in\{1,\ldots,d\}\backslash\{k\}}
    \partial_{x_ix_i}, \quad \text{where} \quad \partial_{x_ix_i} := \frac{\partial^2 }{\partial^2_{x_i}}.
\]
We note that all projectors are commutative, cf.~\cite{T:2017MPMG}.
Using this notation, we can extend Theorem~\ref{theo:H2H41D} to an arbitrary number of dimensions.

\begin{theorem}
Let $d\in \mathbb{N}$ and $p\in \mathbb{N}$ with $p\geq 3$. Then, there exits a constant $c>0$ such that
\[
|u - \mathbf{\Pi}_p u|_{H^2((0,1)^d)} \leq c\,h^2 |u|_{H^4((0,1)^d)}\quad \foralls u\in H^4((0,1)^d)\cap H^1_0((0,1)^d).
\]
\end{theorem}
\begin{proof}
Assume first that $u\in C^\infty((0,1)^d) \cap H^1_0((0,1)^d)$. We prove the statement of the theorem using induction with respect to $d$. Assume that it holds for $d-1$, that is,
\begin{equation}
\label{eq:IndHyp}
    \|\Delta^{x_k} (u-\mathbf{\Pi}^{x_k}_p u)\| \leq c\, h^2 \| \Delta^{x_k} \Delta^{x_k} u\| \quad \text{for} \quad k = 1,\ldots,d.
\end{equation}
 Here and in what follows, all norms are $L^2((0,1)^d)$-norms unless stated otherwise. 
 Now, we show that the statement holds also for $d$. By using the fact that $\mathbf{\Pi}_p$ minimizes the $H^2$-semi-norm (Laplace norm) and $\|\Delta  u\|^2 \leq d \sum^d_{j=1}\|\partial_{x_j x_j}u\|^2$, we get
\begin{align*}
    \|\Delta (u-\mathbf{\Pi}_p u )\|^2 \leq \frac{1}{d}\sum^d_{k=1} \|\Delta (u-\mathbf{\Pi}^{x_k}_p \Pi^{x_k}_p u)\|^2 \leq \sum^d_{k=1} \sum^d_{j=1}\|\partial_{x_j x_j} (u-\mathbf{\Pi}^{x_k}_p \Pi^{x_k}_p u)\|^2.
\end{align*}
We separate this into two groups: $j=k$ and $j\neq k$. We start with $j=k$. Using the triangle inequality, the commutativity of the two projectors, Theorem~\ref{theo:H2H41D}, the $H^2$-stability of $\Pi^{x_k}_p$, and the fact that that $\partial_{x_k x_k}$ and $\mathbf{\Pi}^{x_k}_p$ are commutative, we obtain
\begin{align*}
\sum^d_{k=1}\| \partial_{x_k x_k} (u-\mathbf{\Pi}^{x_k}_p \Pi^{x_k}_p u)\|^2&\leq 2\sum^d_{k=1}\left(\| \partial_{x_k x_k} (u- \Pi^{x_k}_p u)\|^2 + \| \partial_{x_k x_k} \Pi^{x_k}_p (u-\mathbf{\Pi}^{x_k}_p u)\|^2\right)\\
&\leq 2\sum^d_{k=1}\left( 4 h^4\| \partial_{x_k x_k x_k x_k} u\|^2 + \| (I-\mathbf{\Pi}^{x_k}_p ) \partial_{x_k x_k} u\|^2\right).
\end{align*}
Now, we use Theorem~\ref{theo:L2H2dD} to obtain
\begin{align}
\label{eq:proofjisk}
\sum^d_{k=1}\| \partial_{x_k x_k} (u-\mathbf{\Pi}^{x_k}_p \Pi^{x_k}_p u)\|^2 \leq c\, h^4\sum^d_{k=1}\left(\| \partial_{x_k x_k x_k x_k} u\|^2 + \|\Delta^{x_k} \partial_{x_k x_k} u\|^2\right).
\end{align}
For the second group ($j\neq k$), we use the triangle inequality, the induction hypothesis (\ref{eq:IndHyp})
and the $H^2$-stability of $\mathbf{\Pi}^{x_k}$ to obtain
\begin{align*}
\sum_{j\neq k}\|\partial_{x_j x_j} (u-\mathbf{\Pi}^{x_k}_p \Pi^{x_k}_p u)\|^2 &\leq 2\sum_{j\neq k}  \left( \|\partial_{x_j x_j} (u-\mathbf{\Pi}^{x_k}_p  u)\|^2 +\|\partial_{x_j x_j} \mathbf{\Pi}^{x_k}_p(u-\Pi^{x_k}_p  u)\|^2\right)\\
&\leq 2\left( \| \Delta^{x_k}(u-\mathbf{\Pi}^{x_k}_p  u)\|^2 +\| \Delta^{x_k} \mathbf{\Pi}^{x_k}_p(u-\Pi^{x_k}_p  u)\|^2\right)\\
&\leq c\left( h^4\|\Delta^{x_k} \Delta^{x_k}u )\|^2 +\| (I-\Pi^{x_k}_p ) \Delta^{x_k}  u\|^2\right).
\end{align*}
Again, we use Theorem~\ref{theo:H2H41D} and obtain
\begin{equation*}
    \sum^d_{k=1} \sum_{j\neq k}\|\partial_{x_j x_j} (u-\mathbf{\Pi}^{x_k}_p \Pi^{x_k}_p u)\|^2 \leq c\, h^4 \sum^d_{k=1}\left(\|\Delta^{x_k} \Delta^{x_k}u )\|^2 + \| \partial_{x_k x_k} \Delta^{x_k} u\|^2 \right). 
\end{equation*}
Combining this with (\ref{eq:proofjisk}), we finally get
\begin{align*}
    \|\Delta (u-\mathbf{\Pi}_p u )\|^2 &\leq c\, h^4 \sum^d_{k=1}\left(\|\Delta^{x_k} \Delta^{x_k}u )\|^2 + 2\| \partial_{x_k x_k} \Delta^{x_k} u\|^2 + \| \partial_{x_k x_k x_k x_k} u\|^2\right)\\
    &\leq c\, h^4 \| \Delta \Delta u\|^2 \quad \foralls u \in C^\infty((0,1)^d) \cap H^1_0((0,1)^d).
\end{align*}
Note that $\| \Delta \Delta u\| \leq \sqrt{d}\, | u|_{H^4((0,1)^d)} $. Using a standard density argument, we obtain the result also for $u \in H^4((0,1)^d) \cap H^1_0((0,1)^d)$.
\end{proof}
Using interpolation theory (cf.~\cite{bergh2012interpolation}) and the $H^2-H^4$ result above, we obtain a $H^2-H^3$ result, cf. \cite[Theorem 6]{sogn2019robust}.
\begin{theorem}
\label{theo:H2H3dD}
Let $d\in \mathbb{N}$ and $p\in \mathbb{N}$ with $p\geq 3$. Then there exits a constant $c$ such that
\[
|u - \mathbf{\Pi}_p u|_{H^2((0,1)^d)} \leq c\,h |u|_{H^3((0,1)^d)}\quad \foralls u\in H^3((0,1)^d)\cap H^1_0((0,1)^d).
\]
\end{theorem}
We also use interpolation theory and the $L^2-H^2$ result (Theorem~\ref{theo:L2H2dD}) to obtain  a $H^1-H^2$ result.
\begin{theorem}
\label{theo:H1H2dD}
Let $d\in \mathbb{N}$ and $p\in \mathbb{N}$ with $p\geq 3$. Then, there exits a constant $c$ such that
\[
|u - \mathbf{\Pi}_p u|_{H^1((0,1)^d)} \leq c\,h |u|_{H^2((0,1)^d)}\quad \foralls u\in H^2((0,1)^d)\cap H^1_0((0,1)^d).
\]
\end{theorem}
By combining these auxiliary results, we can prove Theorem~\ref{theo:errorEst}:
\begin{proof}
Inequality~(\ref{eq:AppH2H3}) is proven in Theorem~\ref{theo:H2H3dD}.
For the inequality (\ref{eq:AppH1H3}), we combine Theorem~\ref{theo:H1H2dD} and Theorem~\ref{theo:H2H3dD} as follows:
\begin{align*}
    \|\nabla (I-\mathbf{\Pi}_p)u \| = \|\nabla (I-\mathbf{\Pi}_p)(I-\mathbf{\Pi}_p)u \| \leq  c\,h \|\nabla^2 (I-\mathbf{\Pi}_p)u \| \leq c\, h^2 \|\nabla^3 u\|,
\end{align*}
where $\|\cdot\|$ again denotes the $L^2$-norm.
Finally, the inequality~(\ref{eq:AppL2H3}) is proven by combining Theorem~\ref{theo:L2H2dD} and Theorem~\ref{theo:H2H3dD} as follows:
\begin{align*}
    \|(I-\mathbf{\Pi}_p)u \| = \|(I-\mathbf{\Pi}_p)(I-\mathbf{\Pi}_p)u \| \leq  c\,h^2 \|\nabla^2 (I-\mathbf{\Pi}_p)u \| \leq c\, h^3 \|\nabla^3 u\|.
\end{align*}
This concludes the proof.
\end{proof}

\bibliographystyle{siamplain}
\bibliography{bibliography}
\end{document}